\documentclass[10pt]{article}
\usepackage{amssymb,amsmath,amsthm,amsfonts,latexsym,graphicx}
\newcommand{\al}{\alpha}
\newcommand{\be}{\beta}
\newcommand{\de}{\delta}
\newcommand{\eps}{\varepsilon}
\newcommand{\la}{\lambda}
\renewcommand{\phi}{\varphi}
\newcommand{\bN}{\mathbb{N}}
\newcommand{\bR}{\mathbb{R}}

\newcommand{\ph}{\varphi}

\newcommand{\nf}{\infty}
\newcommand{\Om}{\Omega}
\newcommand{\arcsinh}{\operatorname{arcsinh}}
\newcommand{\etaodd}{\eta^{\text{odd}}}
\newcommand{\etaeven}{\eta^{\text{even}}}
\newcommand{\firstres}{\de}
\newcommand{\medstrut}{\vphantom{\int_{0_0}^{1^1}}}
\usepackage{mathtools}
\newcommand{\eqdef}{\coloneqq}

\newtheorem{thm}{Theorem}[section]
\newtheorem{lem}[thm]{Lemma}
\newtheorem{prop}[thm]{Proposition}
\newtheorem{conj}[thm]{Conjecture}
\theoremstyle{definition}
\newtheorem{defn}[thm]{Definition}
\newtheorem{rem}[thm]{Remark}
\newtheorem{numtest}[thm]{Numerical test}
\numberwithin{equation}{section}

\usepackage[colorlinks,allcolors=blue]{hyperref}
\newcommand{\doi}[1]{\href{https://doi.org/#1}{DOI: #1}}
\newcommand{\myurl}[1]{\href{#1}{URL: #1}}

\usepackage{xcolor}

\begin{document}

\title{\vspace{-2cm}{\bf Eigenvalues of even very nice Toeplitz matrices can be unexpectedly erratic}}
\author{Mauricio~Barrera \thanks{Partially supported by CONACYT scholarship.}\ ,
Albrecht~B\"{o}ttcher,\\
Sergei M.~Grudsky \thanks{Partially supported by CONACYT grant 238630.}\ ,
Egor A.~Maximenko \thanks{Partially supported by IPN-SIP projects.}}

\maketitle

\begin{abstract}
It was shown in a series of recent publications
that the eigenvalues of $n\times n$ Toeplitz matrices
generated by  so-called simple-loop symbols
admit certain regular asymptotic expansions into negative powers of $n+1$.
On the other hand, recently two of the authors considered the pentadiagonal Toeplitz matrices
generated by the symbol $g(x)=(2\sin(x/2))^4$, which does not satisfy the simple-loop conditions,
and derived asymptotic expansions of a more complicated form.
We here use these results to show that
the eigenvalues of the pentadiagonal Toeplitz matrices do not admit
the expected regular asymptotic expansion.
This also delivers a counter-example to a conjecture
by Ekstr\"{o}m, Garoni, and Serra-Capizzano
and reveals that the simple-loop condition is essential
for the existence of the regular asymptotic expansion.


\end{abstract}

\maketitle

\section{Main results}

This paper is on the eigenvalues of the $n \times n$ analog $T_n(g)$ of the
symmetric pentadiagonal Toeplitz matrix
\[T_6(g)=\left(\begin{array}{rrrrrr}
6 & -4 & 1 &  &  & \\
-4 & 6 & -4 & 1 &  & \\
1 & -4 & 6 & -4 & 1 & \\
 & 1 & -4 & 6 & -4 & 1\\
 &  & 1 & -4  & 6 & -4\\
 &  &  & 1 & -4 & 6
\end{array}\right).\]
These matrices are generated by the Fourier coefficients of the function
\begin{align}
g(x)
&= e^{-2ix}-4e^{-ix}+6-4e^{ix}+e^{2ix} \notag \\
&= (2-e^{-ix}-e^{ix})^2=(2-2\cos x)^2=\left(2\sin\frac{x}{2}\right)^4. \label{g2}
\end{align}
Previous results, and we will say more about them below, raise the expectation that, given any natural number $p$, the eigenvalues $\lambda_{n,1} < \cdots < \lambda_{n,n}$ of $T_n(g)$
admit an asymptotic expansion
\begin{equation}
\lambda_{n,j}=\sum_{k=0}^p \frac{f_k(\frac{j\pi}{n+1})}{(n+1)^k} +O\left(\frac{1}{(n+1)^{p+1}}\right)\;\:\mbox{as}\;\: n \to \infty\label{basic}
\end{equation}
with the error term being uniform in $1 \le j \le n$ and with continuous functions $f_0, \ldots, f_p:[0,\pi] \to {\bf R}$. The following theorem,
which is the main result of the present paper, shows that this is surprisingly false for $p=4$.

\begin{thm}\label{thm:no_regular_with_n_plus_one}
Let $g$ and $T_n(g)$ be as above. There do not exist continuous functions
$f_0,\ldots,f_4\colon[0,\pi]\to\bR$ and numbers $C>0$, $N\in\bN$
such that
\begin{equation}\label{eq:expansion_4_n_plus_one}
\left|\la_{n,j}-\sum_{k=0}^4 \frac{f_k\left(\frac{j\pi}{n+1}\right)}{(n+1)^k}\right| \le \frac{C}{(n+1)^5}
\end{equation}
for every $n\ge N$ and every $j \in \{1,\ldots,n\}$.
\end{thm}

Unfortunately, there is an unlovely complication. We call it the $n, n+1, n+2$ problem.
In~\eqref{basic} and~\eqref{eq:expansion_4_n_plus_one} we used the denominator $n+1$. This
denominator is very convenient when tackling simple-loop symbols. However, when dealing with
the symbol~\eqref{g2}, the denominator $n+2$ is naturally emerging.
See Remark~\ref{rem:n_plus_two_is_natural}. Therefore we decided to work mostly with $n+2$
in this paper. We will denote the coefficient functions by $f_k$ if the denominator is $n+1$ and by $d_k$
in case it is $n+2$. To avert any confusion, let us state the $n+2$ result we will prove.

\begin{thm}\label{Theo 1.2}
Let $g$ and $T_n(g)$ be as above and let $p \ge 0$ be an integer.

\medskip
{\rm (a)}
There exist continuous functions
$d_0,\ldots,d_p\colon[0,\pi]\to\bR$ and a number $D_p >0$
such that
\begin{equation}\label{12aaa}
\left|\lambda_{n,j}-\sum_{k=0}^p \frac{d_k(\frac{j\pi}{n+2})}{(n+2)^k}\right| \le \frac{D_p}{(n+2)^{p+1}}
\end{equation}
whenever $n\ge 1$ and $\frac{p}{2}\log(n+2) \le j \le n$. These functions $d_0, \ldots, d_p$ are uniquely determined.

\medskip
{\rm (b)}
There is a constant $C>0$ such that
\begin{equation}\label{12ccc}
\left|\la_{n,j}-\sum_{k=0}^3 \frac{d_k\left(\frac{j\pi}{n+2}\right)}{(n+2)^k}\right| \le \frac{C}{(n+2)^4}
\end{equation}
for all $n\ge 1$ and all $j \in \{1,\ldots,n\}$.

\medskip
{\rm (c)}
However, there do not exist numbers $C >0$ and $N\in\bN$ such that
\begin{equation}\label{12bbb}
\left|\la_{n,j}-\sum_{k=0}^4 \frac{d_k\left(\frac{j\pi}{n+2}\right)}{(n+2)^k}\right| \le \frac{C}{(n+2)^5}
\end{equation}
for all $n\ge N$ and all $j \in \{1,\ldots,n\}$.
\end{thm}

In the final section of the paper we will pass from $n+2$ to $n+1$ and prove Theorem~\ref{thm:no_regular_with_n_plus_one}.

Part (b) of Theorem~\ref{Theo 1.2} might suggest that
all eigenvalues $\la_{n,j}$ are \emph{moderately well}
approximated by the sums $\sum_{k=0}^3 d_k(\frac{j\pi}{n+2})/(n+2)^k$.
In fact, as we will show in Remark~\ref{rem:approximation_is_bad},
this approximation is \emph{extremely bad} for the first eigenvalues,
in the sense that the corresponding relative errors do not converge to zero.
However, as Theorem~\ref{Theo 1.2}(a) shows,
asymptotic expansions of the form \eqref{basic} for $p=2,3,4,\ldots$
can be used outside a small neighborhood of the point
at which the symbol has a zero of order greater than $2$.

It is well known that $\la_{n,j}=g(j\pi/n)+O(1/n)$, uniformly in $j$,
implying that~\eqref{basic} and~\eqref{12aaa} hold for $p=0$ with $f_0=d_0=g$.
Figure~\ref{fig:plotg1} shows the plot of the symbol $g$ (from $0$ to $\pi$)
and the eigenvalues of $T_{64}(g)$ as the points $(j\pi/65,\la_{64,j})$ and $(j\pi/66,\la_{64,j})$
with $n+1=65$ and $n+2=66$, respectively.
Notice that the approximation of $\la_{n,j}$ by $g(j\pi/(n+2))$ is not very good for large values of $j$.
It is seen that the approximation of $\la_{n,j}$ by $g(j\pi/(n+1))$ is better.

\begin{figure}[htb]
\centering
\includegraphics{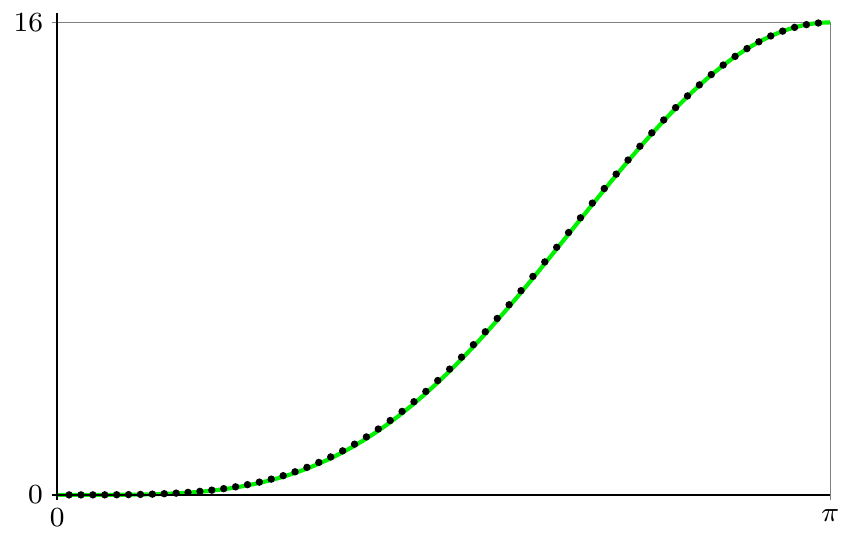}

\vspace{5mm}
\includegraphics{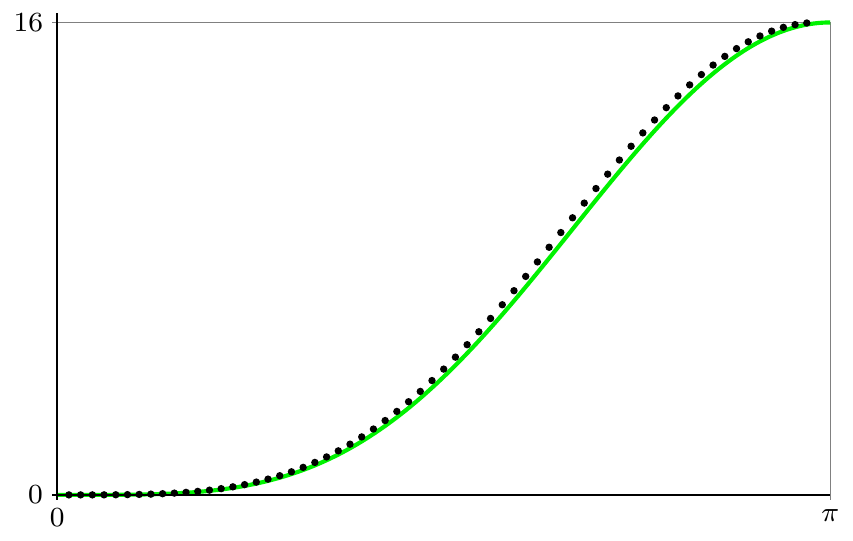}
\caption{Above is the plot of $g$ and the points $(j\pi/65,\la_{64,j})$ for $1\le j\le 64$.
Below we see the plot of $g$ and the points $(j\pi/66,\la_{64,j})$ for $1\le j\le 64$.
\label{fig:plotg1}}
\end{figure}

We will compute the functions $d_1, \ldots, d_4$ of Theorem~\ref{Theo 1.2}.
Knowledge of these functions allows us to illustrate the higher order asymptotics of the eigenvalues
and to depict the expected behavior for $p=0,1,2,3$ and the erratic behavior for $p=4$.
Put
\[
\Omega_{p+1,n,j} \eqdef (n+2)^{p+1}\left(\la_{n,j}-\sum_{k=0}^{p} \frac{d_k(\frac{j\pi}{n+2})}{(n+2)^k}\right).
\]
In Figure~\ref{fig:errors}, we see a perfect matching between $\Om_{p,64,j}$ and $d_p(j\pi/66)$
for $p=1,2,3,4$, except for $p=4$ and $j=1,2$.
The gap between $d_4(\pi/66)$ and $\Om_{4,64,1}$
shows that the asymptotics of $\la_{n,1}$ does not obey the regular rule
with the functions $d_0,d_1,d_2,d_3,d_4$.

\begin{figure}[htb]
\centering
\parbox[t]{0.477\textwidth}{\centering\includegraphics{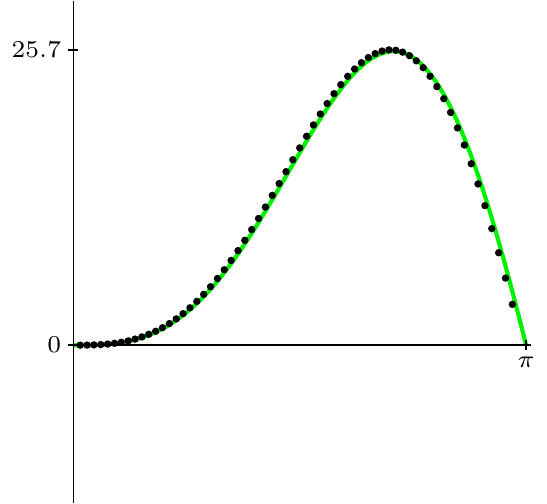}\par (a) $d_1$ and $\Om_{1,64,j}$}
\quad
\parbox[t]{0.477\textwidth}{\centering\includegraphics{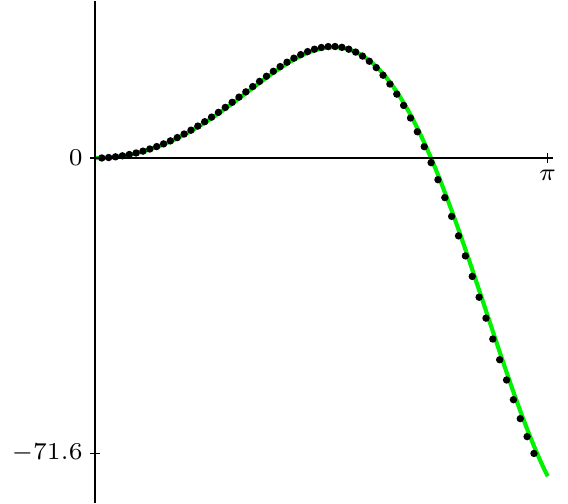}\par (b) $d_2$ and $\Om_{2,64,j}$}
\par\vspace{0.5cm}
\parbox[t]{0.477\textwidth}{\centering\includegraphics{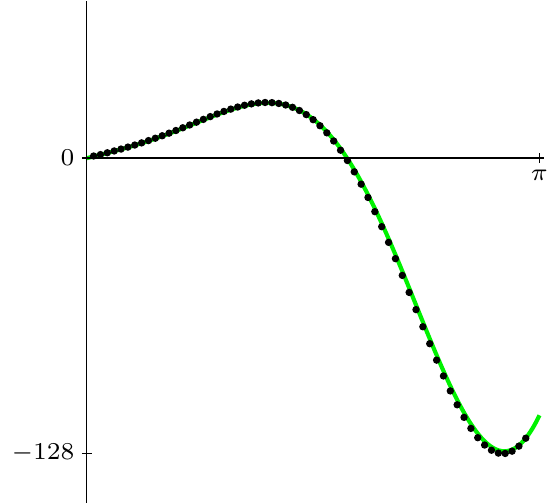}\par (c) $d_3$ and $\Om_{3,64,j}$}
\quad
\parbox[t]{0.477\textwidth}{\centering\includegraphics{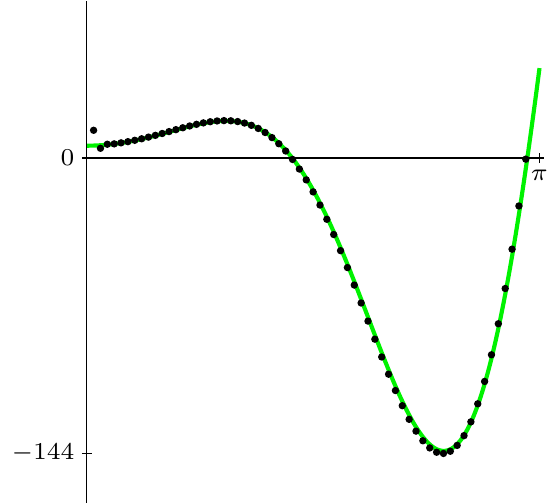}\par (d) $d_4$ and $\Om_{4,64,j}$}
\caption{In subplot (a), we see the graph of $d_1$ and the values of $\Om_{1,64,j}$,
shown as the points $(j\pi/66,\Om_{1,64,j})$.
On subplot (b), we see $d_2$ and $\Om_{2,64,j}$, etc.}
\label{fig:errors}
\end{figure}

Of course, the erratic behavior of the first two eigenvalues in subplot~(d)
of Figure~\ref{fig:errors} might be caused by the circumstance that $n=64$ is not yet large enough.
Figure~\ref{fig:Omega4_1024} reveals that this behavior persists when passing to larger $n$.
In that figure we see the first piece of the graph of $d_4$ and the points $(j\pi/(n+2), \Omega_{4,n,j})$
for $1 \le j \le 64$ and $n=1024$. Now the first three eigenvalues show distinct irregularity.

\begin{figure}[htb]
\centering
\includegraphics{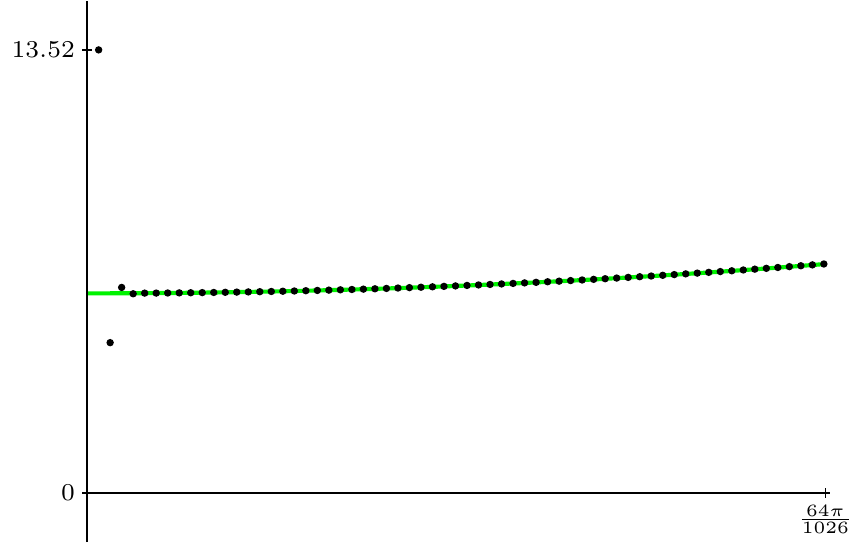}
\caption{The graph of $d_4$ and the points $(j\pi/1026,\Om_{4,1024,j})$ for $j=1,\ldots,64$.\label{fig:Omega4_1024}}
\end{figure}

\medskip
\noindent
Figures~\ref{fig:Omega5_64} and \ref{fig:Omega5_1024} show what happens for $p=5$
and for the matrix dimensions $n=64$ and $n=1024$.

\begin{figure}[htb]
\centering
\includegraphics{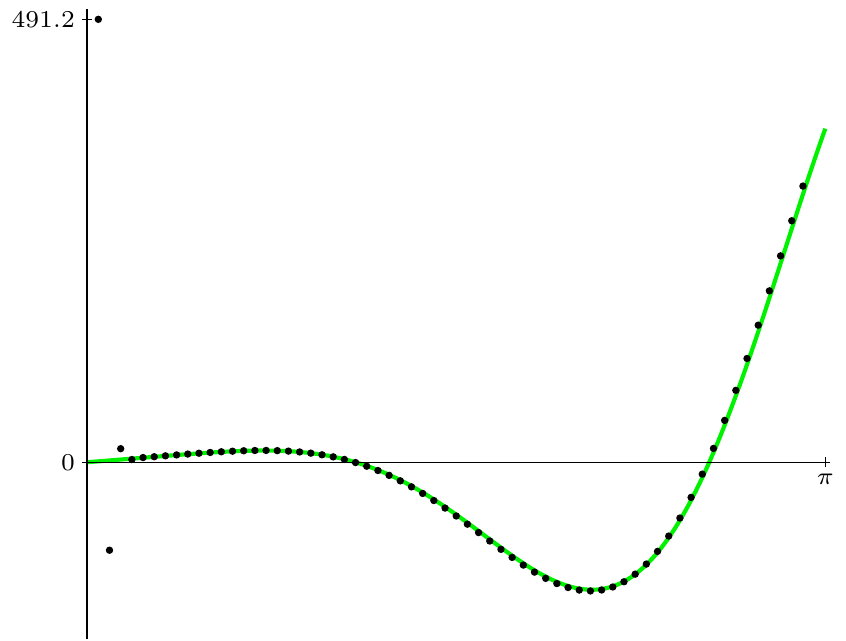}
\caption{The graph of $d_5$ and the points $(j\pi/66,\Om_{5,64,j})$.\label{fig:Omega5_64}}
\end{figure}

\begin{figure}[htb]
\centering
\includegraphics[height=9cm]{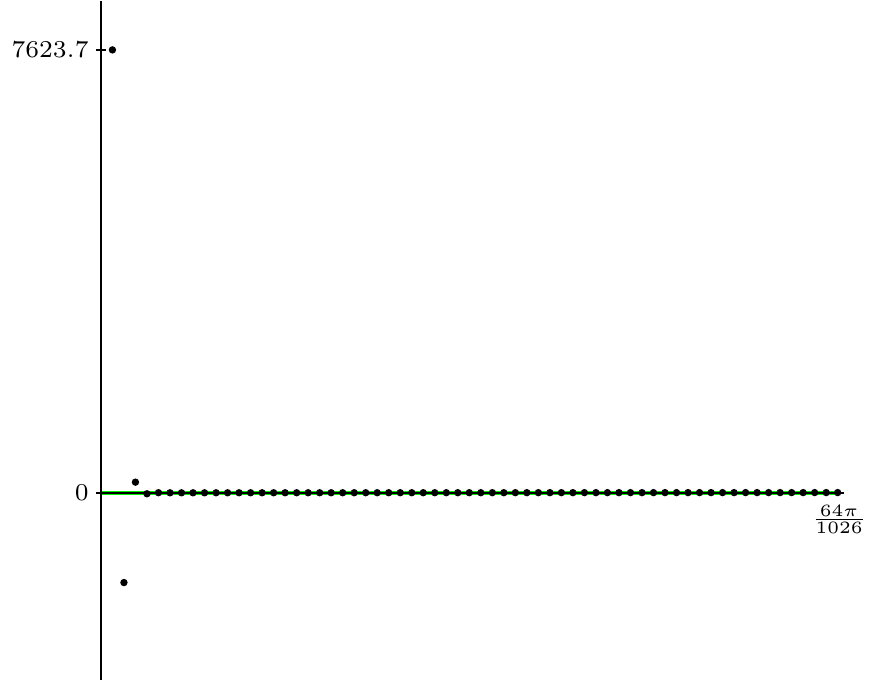}
\caption{The first piece of the graph of $d_5$ (green) and
the points $(j\pi/1026,\Om_{5,1024,j})$ for $j=1,\ldots,64$.
The plot of $d_5$ cannot be distinguished from the abscissa axis.\label{fig:Omega5_1024}}
\end{figure}

\section{Prehistory}

It was the previous papers \cite{BGM2010inside,DIK2012,BBGM2015simpleloop,BGM2017relaxed}
that were devoted to regular asymptotic expansions for the eigenvalues of Toeplitz matrices with so-called simple-loop symbols.
We recall that, in a more general context, the starting point is a $2\pi$-periodic bounded function $g : \bR \to \bR$
with Fourier series $g(x)  \sim \sum_{k=-\infty}^\infty \hat{g}_k e^{ikx}$. The $n \times n$
Toeplitz matrix generated by $g$ is the matrix $T_n(g) =(\hat{g}_{j-k})_{j,k=1}^n$. The function $g$ is referred to as the symbol of the matrix sequence
$\{T_n(g)\}_{n=1}^\infty$.  Examples of simple-loop symbols are even $2\pi$-periodic $C^\infty$ functions $g : \bR \to \bR$
satisfying $g'(x)>0$ for every $x$ in $(0,\pi)$, $g'(0)=0$, $g''(0)>0$, $g'(\pi)=0$, $g''(\pi)<0$.
The requirement that $g$ be a real-valued and even function implies that the matrices $T_n(g)$ are real and symmetric.

In the beginning of Section~7 of \cite{BBGM2015simpleloop},
we also noted
that the mere existence of such regular asymptotic expansions
already helps to approximate the eigenvalues of large matrices
by using the eigenvalues of small matrices and some sort of extrapolation.

Ekstr\"{o}m, Garoni, and Serra-Capizzano \cite{EGS2017areknown}
worked out the idea of such extrapolation in detail. They also emphasized
that the symbols of interest in connection with the discretization of differential equations
are of the form
\begin{equation}
g_m(x)=(2-2\cos x)^m=\left(2\sin\frac{x}{2}\right)^{2m}.\label{g_m}
\end{equation}
In the simplest case $m=1$, the matrices $T_n(g_1)$ are the $n \times n$ analogs
of the tridiagonal Toeplitz matrix
\[T_4(g_1)=\left(\begin{array}{rrrr}
2 & -1 &  & \\
-1 & 2 & -1 &\\
& -1 & 2 & -1\\
& & -1 & 2 \end{array}\right).\]
The eigenvalues of these matrices are known exactly,
\[\la_{n,j}=2-2\cos\frac{j\pi}{n+1}=\left(2\sin\frac{j\pi}{2n+2}\right)^2,\]
and hence they obey the regular asymptotics~\eqref{basic} with $f_0=g$ and $f_k=0$
for $k \ge 1$.
A crucial observation of~\cite{EGS2017areknown}
is that the symbols $g_m$ are no longer simple-loop symbols for $m \ge 2$ , because then the
second derivative at $0$ vanishes.
Our concrete symbol~\eqref{g2} is just $g_2$ and hence not a simple-loop symbol.
Ekstr\"{o}m, Garoni, and Serra-Capizzano
nevertheless conjectured that the regular asymptotic expansions stay true
for smooth even real-valued symbols that are monotone on $[0,\pi]$
and that may have a minimum or a maximum of higher-order.
They verified this conjecture numerically for some examples
and for small values of $p$.
This conjecture has attracted a lot of attention.

Independently and at the same time, two of us~\cite{BG2017pentadiagonal}
considered just the symbol~\eqref{g2} and derived
exact equations and asymptotic expansions for the eigenvalues of $T_n(g)$.
Later, when paper~\cite{EGS2017areknown} came to our attention,
we realized to our surprise that the results of~\cite{BG2017pentadiagonal}
imply that for $g(x)=(2\sin(x/2))^4$  the eigenvalues do not admit
a regular asymptotic expansion of the form~\eqref{basic} with $p=4$.
This is what Theorem~\ref{thm:no_regular_with_n_plus_one} says
and this is a counter-example to the conjecture by Ekstr\"{o}m, Garoni, and Serra-Capizzano.

The rest of the paper is organized as follows.
In Sections~\ref{sec:regular_expansions} and \ref{sec:uniqueness}
we provide some general facts about regular asymptotic expansions.
In Section~\ref{sec:inner_eigenvalues},
using formulas and ideas from \cite{BG2017pentadiagonal},
we show that an analog of \eqref{eq:expansion_4_n_plus_one}
is true for the eigenvalues that are not too close to the minimum of the symbol,
namely, for $2\log(n+2)\le j\le n$,
and provide recipes to compute the corresponding coefficients.
On the other hand, in Section~\ref{sec:first_eigenvalues}
we deduce an asymptotic formula for the first eigenvalue.
In Section~\ref{sec:contradiction}
we prove that the asymptotics from Sections~\ref{sec:inner_eigenvalues}
and \ref{sec:first_eigenvalues} cannot be joined.

\section{Regular expansions of the eigenvalues}
\label{sec:regular_expansions}

In this and the following sections,
we work in abstract settings and use the denominator $n+s$,
where $s$ is an arbitrary positive constant (``shift'').
This allows us to unify the situations with $n+1$ and $n+2$
and to simplify the subsequent references in the last sections of the paper.

We first introduce some notation and recall some facts.
Given a $2\pi$-periodic bounded real-valued function $g$ on the real line,
we denote by $\la_{n,1},\ldots,\la_{n,n}$ the eigenvalues
of the corresponding Toeplitz matrices $T_n(g)$,
ordered in the ascending order: $\la_{n,1}\le\dots\le\la_{n,n}$.
Using the first Szeg\H{o} limit theorem and criteria for weak convergence of probability measures,
we proved in \cite{BBM2016convergence,BBGM2015maximum}
that if the essential range of $g$ is a segment of the real line,
then $\la_{n,j}$ can be uniformly approximated by the values of the quantile function $Q$ (associated to $g$)
at the points $j/(n+s)$:
\begin{equation}\label{eq:quantile_approx}
\max_{1\le j\le n} \left|\la_{n,j}-Q\left(\frac{j}{n+s}\right)\right|= o(1)\quad\text{as}\quad n\to\infty.
\end{equation}
If $g$ is continuous, even, and strictly increasing on $[0,\pi]$, then $Q(x)$ is just $g(\pi x)$.
Denote by $u_{n,j}$ the points of the uniform mesh $j\pi/(n+s)$, $j\in\{1,\ldots,n\}$.
Then \eqref{eq:quantile_approx} can be rewritten in the form
\begin{equation}\label{eq:g_approx}
\max_{1\le j\le n} \left|\la_{n,j}-g(u_{n,j})\right| = o(1)\quad\text{as}\quad n\to\infty.
\end{equation}
Trench proved \cite{Trench1993interlacement}
that for this class of symbols the eigenvalues are all distinct:
\[
g(0)<\la_{n,1}<\dots<\la_{n,n}<g(\pi).
\]
Thus, there exist real numbers $\ph_{n,1},\ldots,\ph_{n,n}$ such that
\[0<\ph_{n,1}<\ldots<\ph_{n,n}<\pi\] and $\la_{n,j}=g(\ph_{n,j})$.
Taking into account \eqref{eq:g_approx},
we can try to use $u_{n,j}$ as an initial approximation for $\ph_{n,j}$.
This approximation can be very inaccurate, but it is better than nothing.

Now let $J$ be an arbitrary set of integer pairs $(n,j)$ such that $1\le j\le n$ for every $(n,j)$ in $J$.
Suppose that for each $(n,j)$ in $J$ the number $\ph_{n,j}$ is the unique solution of an equation
\begin{equation}\label{eq:main_equation}
x = u_{n,j}+\frac{\eta(x)}{n+s}+\rho_{n,j}(x),
\end{equation}
where $\eta$ is an infinitely smooth real-valued function on $[0,\pi]$
and $\{\rho_{n,j}\}_{(n,j)\in J}$
is a family of infinitely smooth real-valued function on $[0,\pi]$ such that
\begin{equation}\label{eq:rho_is_small}
\sup_{0\le x\le \pi}\,\sup_{j: (n,j)\in J} |\rho_{n,j}(x)| = O\left(\frac{1}{(n+s)^p}\right)
\end{equation}
for some  $p$ in $\bN$.

In the simple-loop case, the function $\rho_n$ did not depend on $j$,
and $J$ was of the form $\{(n,j)\colon n\ge N,\ 1\le j\le n\}$ for some $N$.

Let us show how to derive asymptotic expansions of $\ph_{n,j}$ and $\la_{n,j}$
from equation~\eqref{eq:main_equation}.

\begin{prop}\label{prop:regular_asymptotics_from_main_equation}
Let $\eta$ be an infinitely smooth real-valued function on $[0,\pi]$,
and $\{\rho_{n,j}\}_{(n,j)\in J}$ be a family of real-valued functions on $[0,\pi]$
satisfying \eqref{eq:rho_is_small} for
some natural number $p$.
Suppose that for all $(n,j)$ in $J$
equation~\eqref{eq:main_equation} has a unique solution $\ph_{n,j}$.
Then there exists a sequence of real-valued infinitely smooth functions
$c_0,c_1,c_2,\ldots$ defined on $[0,\pi]$ such that
there is a number $r_p>0$ ensuring that, for all $(n,j)$ in $J$,
\begin{equation}\label{eq:regular_asympt_ph}
\left|\ph_{n,j} - \sum_{k=0}^p \frac{c_k(u_{n,j})}{(n+s)^k}\right|\le\frac{r_p}{(n+s)^{p+1}}.
\end{equation}
Furthermore, if $g$ is an infinitely smooth $2\pi$-periodic real-valued even function on $\bR$,
strictly increasing on $[0,\pi]$,
then there exists a sequence of real-valued infinitely smooth functions
$d_0,d_1,c_2,\ldots$ defined on $[0,\pi]$ such that
the numbers $\la_{n,j}\eqdef g(\ph_{n,j})$ can be approximated as follows:
there exists an $R_p$ such that, for all $(n,j)$ in $J$,
\begin{equation}\label{eq:regular_asympt_la}
\left|\la_{n,j} - \sum_{k=0}^p \frac{d_k(u_{n,j})}{(n+s)^k}\right|\le\frac{R_p}{(n+s)^{p+1}}.
\end{equation}
\end{prop}

\begin{proof}
This proposition was essentially proved in \cite{BBGM2015simpleloop,BGM2017relaxed},
with a slightly different notation and reasoning,
including a justification of the fixed-point method.
Here we propose a simpler proof.
Our goal is to show that \eqref{eq:regular_asympt_ph} and \eqref{eq:regular_asympt_la}
are direct and trivial consequences of the main equation \eqref{eq:main_equation}.

In order to simplify notation,
we denote by $O(1/(n+s)^p)$ any expression that may depend on $n$ and $j$
but can be estimated from above by ${C}/{(n+s)^p}$ with $C$ independent of $n$ or $j$.
Then \eqref{eq:main_equation} implies that
\[
\ph_{n,j}=u_{n,j}+O\left(\frac{1}{n+s}\right).
\]
Substitute this expression into \eqref{eq:main_equation}
and expand $\eta$ by Taylor's formula around the point $u_{n,j}$:
\begin{align*}
\ph_{n,j}
&=u_{n,j}+\frac{\eta\left(u_{n,j}+O\left(\frac{1}{n+s}\right)\right)}{n+s}
+O\left(\frac{1}{(n+s)^2}\right)
\\
&=u_{n,j}+\frac{\eta(u_{n,j})}{n+s}+O\left(\frac{1}{(n+s)^2}\right).
\end{align*}
Substituting the last expression into \eqref{eq:main_equation}
and expanding $\eta$ by Taylor formula around $u_{n,j}$ we get
\begin{align*}
\ph_{n,j}
&=u_{n,j}+\frac{\eta\left(u_{n,j}+\frac{\eta(u_{n,j})}{n+s}+O\left(\frac{1}{(n+s)^2}\right)\right)}{n+s}
+O\left(\frac{1}{(n+s)^3}\right)
\\
&=u_{n,j}+\frac{\eta(u_{n,j})}{n+s}
+\frac{\eta(u_{n,j})\eta'(u_{n,j})}{(n+s)^2}+O\left(\frac{1}{(n+s)^3}\right).
\end{align*}
This ``M\"{u}nchhausen trick'' can be applied again and again
(we refer to the story when Baron von M\"{u}nchhausen saved himself
from being drowned in a swamp by pulling on his own hair),
yielding an asymptotic expansion of the form \eqref{eq:regular_asympt_ph} of any desired order $p$.

The first of the functions $c_k$  are
\begin{equation}\label{eq:c_formulas}
\begin{aligned}
&c_0(x)=x, \quad c_1(x)=\eta(x),\quad c_2(x)=\eta(x)\eta'(x), \\
&c_3=\eta (\eta')^2 + \frac{1}{2} \eta^2 \eta'',\quad
c_4=\eta (\eta')^3+\frac{3}{2} \eta^2 \eta' \eta'' +\frac{1}{6}\eta^3\eta'''.\\
&c_5= \eta (\eta')^4 + 3\eta^2 (\eta'^2) \eta''
+\frac{1}{2} \eta^3 (\eta'')^2
+\frac{2}{3} \eta^3 \eta' \eta''' + \frac{1}{24} \eta^4 \eta^{(4)}.
\end{aligned}
\end{equation}
By induction on $p$ it is straightforward to show
that $c_k$ is a uniquely determined polynomial in $ \eta, \eta', \ldots,
\eta^{(k-1)}$ also for $k \ge 6$.

Once we have the asymptotic formulas for $\ph_{n,j}$,
we can use the formula $\la_{n,j}=g(\ph_{n,j})$
and expand the function $g$ by Taylor's formula around the point $u_{n,j}$
to get
\begin{align*}
\la_{n,j}
&=g\left(u_{n,j}+\sum_{k=1}^p \frac{c_k(u_{n,j})}{(n+s)^k}
+ O\left(\frac{1}{(n+s)^{p+1}}\right)\right)
\\
&=\sum_{m=0}^p \frac{g^{(m)}(u_{n,j})}{m!}
\left(\sum_{k=1}^p \frac{c_k(u_{n,j})}{(n+s)^k} + O\left(\frac{1}{(n+s)^{p+1}}\right)\right)^m
\!\!\!\!\!+O(\ph_{n,j}-u_{n,j})^{p+1}.
\end{align*}
Expanding the powers, regrouping the summands,
and writing $\ph_{n,j}-u_{n,j}$ as $O(1/(n+s))$,
we obtain a regular asymptotic formula for $\la_{n,j}$:
\begin{equation}\label{eq:lambda_expansion_with_O}
\la_{n,j}=\sum_{k=0}^p \frac{d_k(u_{n,j})}{(n+s)^k} + O\left(\frac{1}{(n+s)^{p+1}}\right).
\end{equation}
The first of the functions $d_0,d_1,d_2,\ldots$ can be computed by the formulas
\begin{equation}\label{eq:d_formulas}
\begin{aligned}
&d_0 = g,\quad
d_1 = g' c_1,\quad
d_2 = g' c_2+\frac{1}{2} g'' c_1^2,\quad
d_3 = g' c_3 + g'' c_1 c_2 + \frac{1}{6} g''' c_1^3,
\\
&d_4 = g' c_4 + g''\left(c_1 c_3 + \frac{1}{2} c_2^2\right)
+ \frac{1}{2} g''' c_1^2 c_2 + \frac{1}{24} g^{(4)} c_1^4,
\\
&d_5 = g' c_5 + g'' (c_2 c_3 + c_1 c_4) + \frac{1}{2} g''' (c_1^2 c_3 + c_1 c_2^2)\\
&\qquad\quad+ \frac{1}{6} g^{(4)} c_1^3 c_2 + \frac{1}{120} g^{(5)} c_1^5.
\end{aligned}
\end{equation}
It can again be proved by induction on $p$
that the functions $c_0,c_1,c_2,\ldots$ are polynomials
in $\eta,\eta',\eta'',\ldots$
and that the functions $d_0,d_1,d_2,\ldots$ are polynomials
in $c_0,c_1,c_2,\ldots$ and $g,g',g'',\ldots$.
As a consequence, all the functions $c_k$ and $d_k$ are infinitely smooth.
\end{proof}

\begin{rem}
The expressions \eqref{eq:c_formulas} and \eqref{eq:d_formulas}
can be easily derived with various computer algebra systems.
For example, in SageMath we used the following commands
(the expression $1/n$ is denoted by $h$):
\begin{verbatim}
var('u, h, c1, c2, c3, c4, c5'); (eta, g) = function('eta, g')
phiexpansion1 = u + h * eta(u)
phiexpansion2 = u + h * taylor(eta(phiexpansion1), h, 0, 2)
phiexpansion3 = u + h * taylor(eta(phiexpansion2), h, 0, 3)
phiexpansion4 = u + h * taylor(eta(phiexpansion3), h, 0, 4)
phiexpansion5 = u + h * taylor(eta(phiexpansion4), h, 0, 5)
print(phiexpansion5.coefficients(h))
phiformal5 = u + c1*h + c2*h^2 + c3*h^3 + c4*h^4 + c5*h^5
lambdaexpansion5 = taylor(g(phiformal5), h, 0, 5)
print(lambdaexpansion5.coefficients(h))
\end{verbatim}
We also performed similar computations in Wolfram Mathematica, starting with
\begin{verbatim}
phiexpansion0 = u + O[h]
phiexpansion1 = Series[u + h * eta[phiexpansion0], {h, 0, 1}]
\end{verbatim}
\end{rem}

\begin{rem}\label{rem:change_denominator}
If the functions $d_0,d_1,\ldots$ are infinitely smooth,
then one can transform an asymptotic expansion into negative powers of $n+s_1$
into an asymptotic expansion in negative powers of $n+s_2$. For example,
suppose we have
\[
\la_{n,j}=\sum_{k=0}^p \frac{{d}_k(u_{n,j})}{(n+2)^k} + O\left(\frac{1}{(n+2)^{p+1}}\right),
\]
and we want

\vspace{-7mm}
\[
\la_{n,j}=\sum_{k=0}^p \frac{f_k(u_{n,j})}{(n+1)^k} + O\left(\frac{1}{(n+1)^{p+1}}\right).
\]
For $k=0,1$, we have
\begin{align*}
d_k\left(\frac{j\pi}{n+2}\right)
&=d_k\left(\frac{j\pi}{(n+1)\left(1+\frac{1}{n+1}\right)}\right)\\
&= d_k\left(\frac{j\pi}{n+1}-\frac{j\pi}{(n+1)^2}+O\left(\frac{1}{(n+1)^3}\right)\right)\\
&= d_k\left(\frac{j\pi}{n+1}\right)-d_k'\left(\frac{j\pi}{n+1}\right)\frac{j\pi}{n+1}\,\frac{1}{n+1}+O\left(\frac{1}{(n+1)^4}\right),
\end{align*}
and thus

\vspace{-7mm}
\begin{align*}
& d_0\left(\frac{j\pi}{n+2}\right)+d_1\left(\frac{j\pi}{n+2}\right)\frac{1}{n+2}+O\left(\frac{1}{(n+2)^2}\right)\\
& = d_0\left(\frac{j\pi}{n+1}\right)-d_0'\left(\frac{j\pi}{n+1}\right)\frac{j\pi}{n+1}\,\frac{1}{n+1}\\
& \quad +  d_1\left(\frac{j\pi}{n+1}\right)\frac{1}{n+1} + O\left(\frac{1}{(n+1)^2}\right),
\end{align*}
resulting in the equalities $f_0(x)=d_0(x)$ and  $f_1(x)=d_1(x)-xd_0'(x)$.
\end{rem}

\begin{rem}
The hard part of the work in \cite{BBGM2015simpleloop,BGM2017relaxed}
was to derive equation \eqref{eq:main_equation} and an explicit formula for $\eta$,
to verify that $\eta$ is sufficiently smooth,
to establish upper bounds for the functions $\rho_n$,
and to prove that \eqref{eq:main_equation}
has a unique solution for every $n$ large enough and for every $j$.
Moreover, all this work was done under the assumption that $g$
has some sort of smoothness of a finite order.
In Proposition~\ref{prop:regular_asymptotics_from_main_equation}
we just require all these properties.
\end{rem}

\section{Uniqueness of the regular asymptotic expansion}
\label{sec:uniqueness}

As in the previous section, we fix some $s>0$.

If there exists an asymptotic expansion of the form \eqref{eq:lambda_expansion_with_O},
then the functions $d_0,d_1,d_2,\ldots$ are uniquely determined.
Let us state and prove this fact formally.
Instead of requiring \eqref{eq:lambda_expansion_with_O} for all $n$ and $j$,
we assume it holds for a set of pairs $(n,j)$
such that the quotients $u_{n,j}\eqdef j\pi/(n+s)$ ``asymptotically fill'' $[0,\pi]$.
Here is the corresponding technical definition.

\begin{defn}\label{defn:asymptotically_fill_by_quotients}
Let $J$ be a subset of $\bN^2$.
We say that $J$ \emph{asymptotically fills $[0,\pi]$ by quotients}
if for every $x$ in $[0,\pi]$, every $N$ in $\bN$, and every $\de>0$
there is a pair of numbers $(n,j)$ in $J$
such that $n\ge N$, $1\le j\le n$, and $|u_{n,j}-x|\le \de$.
\end{defn}

It is easy to see that $J$ asymptotically fills $[0,\pi]$ by quotients
if and only if the set $\{u_{n,j}\colon\ (n,j)\in J\}$ is dense in $[0,\pi]$.

\begin{prop}\label{prop:uniqueness}
Let $p \ge 0$ be an integer, let
$d_0,d_1,\ldots,d_p$ and $\widetilde{d}_0,\widetilde{d}_1,\ldots,\widetilde{d}_p$
be continuous functions on $[0,\pi]$, let $C>0$,
and let $J$ be a subset of $\bN^2$ such that $J$ asymptotically fills $[0,\pi]$ by quotients.
Suppose that for every pair $(n,j)$ in $J$ the inequalities
\[
\left|\la_{n,j}-\sum_{k=0}^p \frac{d_k(u_{n,j})}{(n+s)^k}\right|\le\frac{C}{(n+s)^{p+1}},\quad
\left|\la_{n,j}-\sum_{k=0}^p \frac{\widetilde{d}_k(u_{n,j})}{(n+s)^k}\right|\le\frac{C}{(n+s)^{p+1}}
\]
hold.
Then $d_k(x)=\widetilde{d}_k(x)$ for every $k \in \{0,\ldots,p\}$ and every $x \in [0,\pi]$.
\end{prop}

\begin{proof}
Denote the function $d_p-\widetilde{d}_p$ by $h_p$.
It is clear that $h_0=0$.
Proceeding by mathematical induction on $p$,
we assume that $h_k$ is the zero constant for every $k$ with $k<p$,
and we have to show that $h_p$ is the zero constant.

Let $x\in[0,\pi]$ and $\eps>0$.
Using the continuity of $h_p$ at the point $x$,
choose $\de>0$ such that $|y-x|\le\de$ implies $|h_p(y)-h_p(x)|\le\eps/2$.
Take $N$ such that $2C/(N+s)\le\eps/2$.
After that, pick $n$ and $j$ such that $(n,j)\in J$, $n\ge N$, and $|u_{n,j}-x|\le\de$.
Then
\[
\left|\frac{d_p(u_{n,j})}{(n+s)^p}-\frac{\widetilde{d}_p(u_{n,j})}{(n+s)^p}\right|
\le\frac{2C}{(n+s)^{p+1}},
\]
which implies
\[
|h_p(u_{n,j})|\le\frac{2C}{n+s}\le\frac{2C}{N+s}\le\frac{\eps}{2}.
\]
Finally,
\[
|h_p(x)| \le |h_p(x)-h_p(u_{n,j})| + |h_p(u_{n,j})| \le \frac{\eps}{2}+\frac{\eps}{2} = \eps.
\]
As $\eps >0$ can be chosen arbitrarily, it follows that $h_p$ is identically zero.
\end{proof}

\section{An example with a minimum of the fourth order}
\label{sec:inner_eigenvalues}

We now consider the pentadiagonal Toeplitz matrices generated by the trigonometric polynomial
\begin{equation}\label{eq:g}
g(x)=\left(2\sin\frac{x}{2}\right)^4.
\end{equation}
The function $g$ takes real values, is even, and strictly increases on $[0,\pi]$.
Nevertheless, $g$ does not belong to simple-loop class,
because $g$ has a minimum of the fourth order:
$g(0)=g'(0)=g''(0)=g'''(0)=0$, $g^{(4)}(0)>0$.

The purpose of this section is to recall some results of \cite{BG2017pentadiagonal}
and to derive some new corollaries.
We begin by introducing some auxiliary functions:
\begin{align*}
& \be(x)\eqdef 2\arcsinh\left(\sin\frac{x}{2}\right)
=2\ln\left(\sin\frac{x}{2}+\sqrt{1+\left(\sin\frac{x}{2}\right)^2}\right),\\
& f(x)\eqdef \be'(x)=\frac{\cos\frac{x}{2}}{1+\left(\sin\frac{x}{2}\right)^2},\\
& \etaodd_n(x) \eqdef 2\arctan\left(\frac{1}{f(x)}\coth\frac{(n+2)\be(x)}{2}\right),\\
& \etaeven_n(x) \eqdef 2\arctan\left(\frac{1}{f(x)}\tanh\frac{(n+2)\be(x)}{2}\right),\\
& \eta_{n,j}(x) \eqdef
\begin{cases}
\etaodd_n(x), & \text{if}\ j\ \text{is odd},\\
\etaeven_n(x), & \text{if}\ j\ \text{is even}.
\end{cases}
\end{align*}
As previously, we denote by $\ph_{n,j}$ the points in $(0,\pi)$ such that $\la_{n,j}=g(\ph_{n,j})$.
In this example, we let $u_{n,j}$ stand for $j\pi/(n+2)$.

In \cite[Theorems 2.1 and 2.3]{BG2017pentadiagonal},
two of us used Elouafi's formulas \cite{Elouafi2014} for the determinants of Toeplitz matrices
and derived exact equations for the eigenvalues of $T_n(g)$.
Namely, it was proved that there exists an $N_0$
such that if $n\ge N_0$ and $j\in\{1,\ldots,n\}$,
then $\ph_{n,j}$ is the unique solution in the interval $(u_{n,j},u_{n,j+1})$ of the equation
\begin{equation}\label{eq:main_pentadiagonal}
x = u_{n,j} + \frac{\eta_{n,j}(x)}{n+2}.
\end{equation}
The corresponding equation in~\cite{BG2017pentadiagonal}
is written in a slightly different (but equivalent) form,
without joining the cases of odd and even values of $j$.

Equation~\eqref{eq:main_pentadiagonal} is hard to derive but easy to verify numerically.
We computed the eigenvalues by general numerical methods in Wolfram Mathematica,
using high-precision arithmetic with $100$ decimal digits after the floating point,
and obtained  coincidence in \eqref{eq:main_pentadiagonal}
up to $99$ decimal digits for each $n$ from $10$ to $100$ and for each $j$ from $1$ to $n$.

Equation~\eqref{eq:main_pentadiagonal} is more complicated than \eqref{eq:main_equation},
in the sense that now instead of one function $\eta$
we have a family of functions, depending on $n$ and on the parity of $j$.

Notice that if $x$ is not too close to zero and $n$ is large enough,
then $\be(x)$ is not too close to zero,
the product $\frac{n+2}{2}\be(x)$ is large
and the expressions $\tanh\frac{(n+2)\be(x)}{2}$
and $\coth\frac{(n+2)\be(x)}{2}$ are very close to $1$.
Denote by $\eta$ the function obtained from $\etaodd_n$ and $\etaeven_n$
by neglecting these expressions, that is,
\begin{equation}\label{eq:eta_pentadiagonal}
\eta(x)\eqdef 2\arctan\left(\frac{1}{f(x)}\right),
\end{equation}
and put
\[
\rho_{n,j}(x)\eqdef\frac{\eta_{n,j}(x)-\eta(x)}{n+2}.
\]
Then the main equation \eqref{eq:main_pentadiagonal} takes the form \eqref{eq:main_equation} with $s=2$:
\begin{equation}\label{eq:main_equation_pentadiagonal}
x = u_{n,j} + \frac{\eta(x)}{n+2} + \rho_{n,j}(x).
\end{equation}
So, for each $(n,j)$ in $J_0$ the number $\ph_{n,j}$ is the unique solution of
\eqref{eq:main_equation_pentadiagonal} in the interval $(u_{n,j},u_{n,j+1})$.

Figure~\ref{fig:eta} shows that the functions $\etaodd_{64}$, $\etaeven_{64},$
and $\eta$ almost coincide outside a small neighborhood of zero.

\begin{figure}[ht]
\centering
\includegraphics[height=7.5cm]{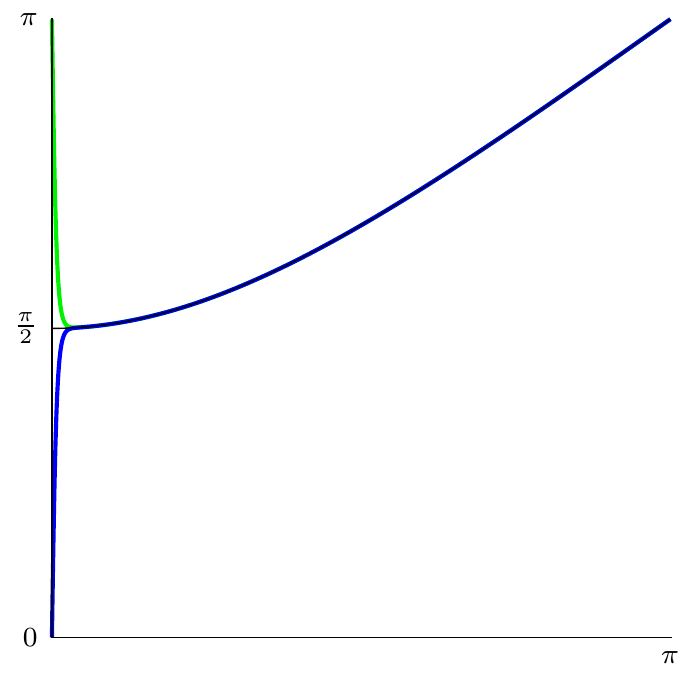}
\caption{Plots of $\etaodd_{64}$ (green), $\etaeven_{64}$ (blue), and $\eta$ (black)}
\label{fig:eta}
\end{figure}

The following lemma provides us with upper estimates for $\rho_{n,j}(x)$.

\begin{lem}\label{lem:rho_upper_estimate}
Let $n,j\in\bN$. If $1\le j\le n/2$, then
\begin{equation}\label{eq:rho_upper_estimate_j_small}
\sup_{u_{n,j}\le x\le u_{n,j+1}} |\rho_{n,j}(x)| \le \frac{6e^{-2j}}{n+2}.
\end{equation}
If $n/2\le j\le n$, then
\begin{equation}\label{eq:rho_upper_estimate_j_large}
\sup_{u_{n,j}\le x\le u_{n,j+1}} |\rho_{n,j}(x)| \le 6e^{-(n+2)\pi/2}.
\end{equation}
\end{lem}

\begin{proof}
First suppose that $1\le j\le n/2$ and $u_{n,j}\le x\le u_{n,j+1}$.
Then
\[
\frac{j\pi}{n+2}\le x\le \frac{(j+1)\pi}{n+2} \le\frac{\pi}{2}.
\]
It is readily verified that $\be(x)\ge 2x/\pi$ for every $x$ in $[0,\pi/2]$.
Consequently,
\[
\frac{(n+2)\be(x)}{2} \ge j.
\]
It is also easy to see that
$0\le 1-\tanh(y)\le 2e^{-2y}$
and
$0\le \coth(y)-1\le 3e^{-2y}$ for $y\ge 1$,
$f(x)>1/2$ for $x$ in $[0,\pi/2]$,
and that $\arctan$ is Lipschitz continuous with coefficient $1$.
Thus
\[
|\eta_{n,j}(x)-\eta(x)|\le 6e^{-2j},
\]
which yields \eqref{eq:rho_upper_estimate_j_small}.

Now consider the case $n/2\le j\le n$.
Here we use the estimates $\be(x)\ge x/2$ and $f(x)>1/(n+2)$ to obtain
\[
\frac{(n+2)\be(x)}{2}\ge\frac{(n+2)\pi}{4},
\]
\[
|\eta_{n,j}(x)-\eta(x)|\le 6(n+2)e^{-(n+2)\pi/2},
\]
which results in \eqref{eq:rho_upper_estimate_j_large}.
\end{proof}

The next proposition is similar to Theorem~2.3 from \cite{BG2017pentadiagonal},
but here we join the cases of odd and even values of $j$
and get rid of the additional requirement that $n\ge N_0$.
We use essentially the same arguments to prove the existence of the solution,
but a simpler argument to prove the uniqueness.

\begin{prop}\label{prop:existance_and_uniqueness}
For all $n\ge 1$ and all $j \in \{1,\ldots,n\}$,
the number $\ph_{n,j}$ is the unique solution of the equation \eqref{eq:main_pentadiagonal}
on the interval $(u_{n,j},u_{n,j+1})$.
\end{prop}

\begin{proof}
Let $n\ge 1$.
For each $j \in \{1,\ldots,n\}$,
the main equation can be written in the form
\begin{equation}\label{eq:main_left}
(n+2)x-\eta_{n,j}(x)=j\pi.
\end{equation}
By Theorem~2.1 from \cite{BG2017pentadiagonal},
if $x$ belongs to $(0,\pi)$ and satisfies \eqref{eq:main_left} for some integer $j$,
then the number $g(x)$ is an eigenvalue of $T_n(g)$.

Notice that $f(x)>0$ and $\be(x)>0$ for every $x \in (0,\pi)$.
Using the definitions of $\tanh$, $\coth$, and $\arctan$,
we conclude that $0<\eta_{n,j}(x)<\pi$ for each $x \in (0,\pi)$;
see also Figure~\ref{fig:eta}.
Denote the left-hand side of \eqref{eq:main_left} by $F_{n,j}(x)$.
Then
\begin{align*}
F_{n,j}(u_{n,j}) &= j\pi-\eta_{n,j}(u_{n,j})<j\pi,\\
F_{n,j}(u_{n,j+1}) &= (j+1)\pi-\eta_{n,j}(u_{n,j+1})>j\pi.
\end{align*}
Hence, by the intermediate value theorem, equation \eqref{eq:main_left}
has at least one solution in the interval $(u_{n,j},u_{n,j+1})$.
At this moment we do not know whether this solution is unique.
So let us, for each $j$, denote by $\psi_{n,j}$ one of the solutions of \eqref{eq:main_left} on $(u_{n,j},u_{n,j+1})$.

Contrary to what we want, assume that for some $j \in \{1,\ldots,n\}$
equation \eqref{eq:main_left} has another solution $x$ belonging to $(u_{n,j},u_{n,j+1})$.
The $n+1$ numbers $\psi_{n,1},\ldots,\psi_{n,n},x$ are different.
Since $g$ is strictly increasing on $[0,\pi]$,
the corresponding eigenvalues $g(\psi_{n,1}),\ldots,g(\psi_{n,n}),g(x)$ are different, too.
This contradicts the fact that the matrix $T_n(g)$ has only $n$ eigenvalues.

We conclude that for each $j$ equation~\eqref{eq:main_left} has only one solution $\psi_{n,j}$ in $(u_{n,j},u_{n,j+1})$.
The numbers $\psi_{n,j}$ satisfy $\psi_{n,1}<\ldots<\psi_{n,n}$,
and their images under $g$ are eigenvalues of $T_n(g)$,
so $g(\psi_{n,j})=\la_{n,j}$ and $\psi_{n,j}=\ph_{n,j}$ for all $j$.
\end{proof}

The next proposition gives asymptotic formulas
for the eigenvalues $\la_{n,j}$ provided $j$ is ``not too small''.
It mimics Theorem~2.6 from \cite{BG2017pentadiagonal},
the novelty being that we here join the cases of odd and even values of $j$
and state the result for an arbitrary order $p$.

\begin{prop}\label{prop:inner_eigenvalues}
For every $p \in \bN$, the functions $\rho_{n,j}$ admit the asymptotic upper estimate
\begin{equation}\label{eq:rho_is_exp_small}
\max_{(p/2)\log(n+2)\le j\le n}\;\sup_{x\in[u_{n,j},u_{n,j+1}]} |\rho_{n,j}(x)| = O\left(\frac{1}{n^{p+1}}\right).
\end{equation}
Moreover, for every $p \in \bN$, every $n \in \bN$, and every $j$ satisfying
\begin{equation}\label{eq:j_ge_log}
\frac{p}{2}\log(n+2) \le j \le n,
\end{equation}
the numbers $\ph_{n,j}$ and $\la_{n,j}$ have asymptotic expansions of the form
\begin{align}
\label{eq:ph_expansion_pentadiagonal_far_from_0}
\ph_{n,j}&=\sum_{k=0}^p \frac{c_k(u_{n,j})}{(n+2)^k} + O\left(\frac{1}{(n+2)^{p+1}}\right),
\\
\label{eq:la_expansion_pentadiagonal_far_from_0}
\la_{n,j}&=\sum_{k=0}^p \frac{d_k(u_{n,j})}{(n+2)^k} + O\left(\frac{1}{(n+2)^{p+1}}\right),
\end{align}
where the upper estimates of the residue terms are uniform in $j$,
the functions $c_k$ and $d_k$ are infinitely smooth
and can be expressed in terms of $\eta$ and $g$
by the formulas shown in the proof of Proposition~\ref{prop:regular_asymptotics_from_main_equation}.
\end{prop}

\begin{proof}
We have to verify the upper bound \eqref{eq:rho_is_exp_small}.
The other statements then follow from Proposition~\ref{prop:regular_asymptotics_from_main_equation}.
Let $p,n\in\bN$ and $j$ satisfy \eqref{eq:j_ge_log}.
If $j\le n/2$, then \eqref{eq:rho_upper_estimate_j_small} gives
\[
\frac{6e^{-2j}}{n+2}\le \frac{6e^{-p\log(n+2)}}{n+2}=\frac{6}{(n+2)^{p+1}},
\]
while if $j>n/2$, we obtain from \eqref{eq:rho_upper_estimate_j_large} that
\[
e^{-(n+2)\pi/2}=O\left(\frac{1}{n^{p+1}}\right).
\]
Joining these two cases we arrive at \eqref{eq:rho_is_exp_small}.
\end{proof}

In Proposition~\ref{prop:regular_asymptotics_from_main_equation}
we expressed the first of the coefficients $c_k$ and $d_k$
in terms of the first derivatives of $g$ and $\eta$.
Here are explicit formulas for $g',\ldots,g^{(5)}$:
\begin{equation}\label{eq:g_derivatives}
\begin{aligned}
g'(x) &= 23\cos\frac{x}{2}\left(\sin\frac{x}{2}\right)^3,
&
g''(x) &= 16 (1+2\cos(x)) \left(\sin\frac{x}{2}\right)^2,
\\
g'''(x) &= - 8\sin(x)+16\sin(2x),
&
g^{(4)}(x) &= -8\cos(x)+32\cos(2x),
\\
g^{(5)}(x) &= 8\sin(x)-64\sin(2x).
\end{aligned}
\end{equation}
For $\eta',\ldots,\eta^{(4)}$ we have
\begin{equation}\label{eq:eta_derivatives}
\begin{aligned}
\eta'(x) &= \frac{\sin\frac{x}{2}}{\left(1 + \left(\sin\frac{x}{2}\right)^2\right)^{1/2}},
&
\eta''(x) &= \frac{\sqrt{2}\cos\frac{x}{2}}{\left(3-\cos(x)\right)^{3/2}},
\\
\eta'''(x) &= -\frac{5\sin\frac{x}{2}+\sin\frac{3x}{2}}{\sqrt{2} (3 - \cos(x))^{5/2}}, &
\eta^{(4)}(x) &=
\frac{-4\cos\frac{x}{2}+19\cos\frac{3x}{2}+\cos\frac{5x}{2}}%
{2\sqrt{2} (3 - \cos(x))^{7/2}}.
\end{aligned}
\end{equation}

\begin{numtest} \label{numtest:inner_eigenvalues}
If order to test \eqref{eq:la_expansion_pentadiagonal_far_from_0} numerically for $p=4$,
we computed $g',\ldots,g^{(4)}$ by \eqref{eq:g_derivatives},
$\eta,\eta',\ldots,\eta^{(3)}$ by \eqref{eq:eta_pentadiagonal} and \eqref{eq:eta_derivatives},
$c_0,c_1,\ldots,c_4$ by \eqref{eq:c_formulas} and $d_0,d_1,\ldots,d_4$ by \eqref{eq:d_formulas}.
The exact eigenvalues were computed by simple iteration in equation~\eqref{eq:main_equation_pentadiagonal}
and independently by general eigenvalue algorithms (for $n\le 1024$).
All computations were made in high-precision arithmetic with $100$ decimal digits after the floating point,
in SageMath and independently in Wolfram Mathematica.
Denote by $E_{n,4}$ the maximal error in \eqref{eq:la_expansion_pentadiagonal_far_from_0}, with $p=4$:
\[
E_{n,4} \eqdef \max_{2\log(n+2)\le j\le n}
\left|\la_{n,j}-\sum_{k=0}^{4} \frac{d_k(u_{n,j})}{(n+2)^k}\right|.
\]
The following table shows $E_{n,4}$ and $(n+2)^5 E_{n,4}$
for various values of $n$.
\[
\begin{array}{c|c|c|c|c|c}
\medstrut &
n=64 &
n=256 &
n=1024 &
n=4096 &
n=16384
\\\hline
\medstrut
E_{n,4} &
2.4\cdot 10^{-7} &
3.1\cdot 10^{-10} &
3.2\cdot 10^{-13} &
3.2\cdot 10^{-16} &
3.1\cdot 10^{-19}
\\\hline
\medstrut
(n+2)^5 E_{n,4} &
306.72 &
354.87 &
366.61 &
369.52 &
370.25
\end{array}
\]
We see that the numbers $E_{n,4}$ really behave like $O(1/(n+2)^5)$.
\end{numtest}

\section{An asymptotic formula for the first eigenvalues in the example}
\label{sec:first_eigenvalues}

In this section we study the asymptotic behavior of $\la_{n,j}$
as $n$ tends to $\infty$, considering $j$ as a fixed parameter.

Using the definition of $\arctan$ and the formula for
$\tan(x+j\pi/2)$, we can rewrite equation~\eqref{eq:main_pentadiagonal}
in the equivalent form
\begin{equation}\label{eq:exact_equation_first_ph}
f(x)^{(-1)^{j+1}} \tanh\frac{(n+2)\be(x)}{2}=(-1)^j \tan\frac{(n+2)x}{2}.
\end{equation}
The first factor on the left-hand side of \eqref{eq:exact_equation_first_ph}
is just $f(x)$ for odd values of $j$ and $1/f(x)$ for even values of $j$.
We know that
\[
\frac{j\pi}{n+2}\le\ph_{n,j}\le\frac{(j+1)\pi}{n+2},
\]
and it is natural to expect that the product $(n+2)\ph_{n,j}$
has a finite limit $\al_j$ as $n$ tends to infinity and $j$ is fixed.
Assuming this and taking into account that
\[
f(x)\to 1,\quad \be(x)\sim x,\quad\text{as}\quad x\to0,
\]
we can pass to the limit in \eqref{eq:exact_equation_first_ph}
to obtain a simple transcendental equation for $\al_j$.
This is an informal motivation of the following formal reasoning.

For each $j$ in $\bN$, denote by $\al_j$ the unique real number
that belongs to the interval $(j\pi,(j+1)\pi)$ and satisfies
\begin{equation}\label{eq:firstcoef_equation}
\tanh\frac{\al_j}{2}=(-1)^j \tan\frac{\al_j}{2}.
\end{equation}
Figure \ref{fig:transcendental} shows both sides of \eqref{eq:firstcoef_equation} for $j=1,2,3$.
\begin{figure}[ht]
\centering
\includegraphics{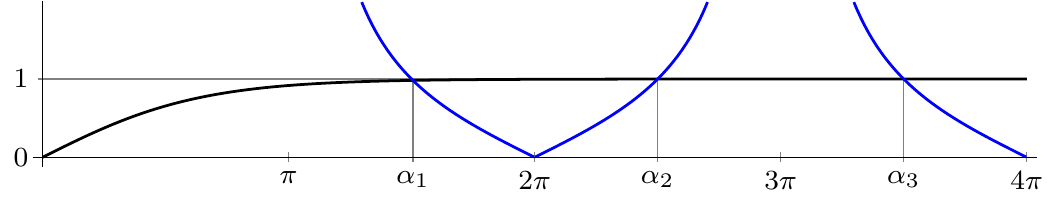}
\caption{The left-hand side (black) and the right-hand side (blue) of \eqref{eq:firstcoef_equation},
for $j=1$ on $(\pi,2\pi)$, for $j=2$ on $(2\pi,3\pi)$ and for $j=3$ on $(3\pi,4\pi)$.
\label{fig:transcendental}}
\end{figure}

For each $j$, the transcendental equation \eqref{eq:firstcoef_equation} is easy to solve by numerical methods.
Approximately,
\[
\al_1\approx 4.73004,\qquad
\al_2\approx 7.85320,\qquad
\al_3\approx 10.99561.
\]
It follows from \eqref{eq:firstcoef_equation}  that
$\al_j>\frac{(2j+1)\pi}{2}$ if $j$ is odd
and $\al_j<\frac{(2j+1)\pi}{2}$ if $j$ is even.
In particular,
\[
\al_1>\frac{3\pi}{2},\qquad
\al_2<\frac{5\pi}{2},\qquad
\al_3>\frac{7\pi}{2}.
\]%
We remark that differences between $\al_j$ and $(2j+1)\pi/2$ are extremely small:
\[
\al_1-\frac{3\pi}{2}\approx 1.8\cdot 10^{-2},\quad
\al_2-\frac{5\pi}{2}\approx -7.8\cdot 10^{-4},\quad
\al_3-\frac{7\pi}{2}\approx 3.3\cdot 10^{-5}.
\]%

Contrary to the general agreement of this paper,
the upper estimates of the residual terms
in the following proposition are not uniform in $j$.
Thus we use the notation $O_j$ instead of $O$.

\begin{prop}\label{prop:first_eigenvalues}
Let $g$ be the function defined by \eqref{eq:g}
and define $\ph_{n,j} \in (0,\pi)$ by $\la_{n,j}=g(\ph_{n,j})$.
Then for each fixed $j$ in $\bN$,
$\ph_{n,j}$ and $\la_{n,j}$ satisfy the asymptotic formulas
\begin{align}
\label{eq:first_ph_asymptotics}
\ph_{n,j}&=\frac{\al_j}{n+2}+O_j\left(\frac{1}{(n+2)^3}\right),
\\
\label{eq:first_eigenvalues_asymptotics}
\la_{n,j} &= \left(\frac{\al_j}{n+2}\right)^4 + O_j\left(\frac{1}{(n+2)^6}\right).
\end{align}
\end{prop}

\begin{proof}
Fix $j$ in $\bN$.
We are going to treat \eqref{eq:exact_equation_first_ph} by asymptotic methods, as $n$ tends to infinity.
Put
\[
\firstres_{n,j} \eqdef (n+2) \ph_{n,j} - \al_j,
\]
i.e., represent the product $(n+2)\ph_{n,j}$ in the form
\[
(n+2)\ph_{n,j}=\al_j+\firstres_{n,j}.
\]
It is easy to verify that, as $x\to 0$,
\[
f(x) = 1 + O(x^2),\qquad
\be(x) = x + O(x^3).
\]
Moreover, we know that
$\frac{j\pi}{n+2}\le\ph_{n,j}\le\frac{(j+1)\pi}{n+2}$ and thus $\ph_{n,j}=O_j(1/(n+2))$.
Therefore
\begin{align*}
& f(\ph_{n,j})=1+O_j\left(\frac{1}{(n+2)^2}\right),\qquad
\frac{1}{f(\ph_{n,j})}=1+O_j\left(\frac{1}{(n+2)^2}\right),\\
& \frac{(n+2)}{2}\be(\ph_{n,j})
=\frac{\al_j+\firstres_{n,j}}{2}+O_j\left(\frac{1}{(n+2)^2}\right),\\
& \tanh\frac{(n+2)}{2}\be(\ph_{n,j})
=\tanh\frac{\al_j+\firstres_{n,j}}{2}+O_j\left(\frac{1}{(n+2)^2}\right).
\end{align*}
By the mean value theorem, there exist some numbers $\xi_{1,n,j}$ and $\xi_{2,n,j}$
between $\al_j/2$ and $(\al_j+\firstres_{n,j})/2$ such that
\[
\tanh\frac{\al_j+\firstres_{n,j}}{2}-\tanh\frac{\al_j}{2}
=\tanh'(\xi_{1,n,j})\frac{\firstres_{n,j}}{2}
\]
and
\[
\tan\frac{\al_j+\firstres_{n,j}}{2}-\tan\frac{\al_j}{2}
=\tan'(\xi_{2,n,j})\frac{\firstres_{n,j}}{2}.
\]
After replacing $x$ by $\ph_{n,j}$, equation \eqref{eq:exact_equation_first_ph} takes the form
\begin{align*}
& \tanh\frac{\al_j}{2}
+\tanh'(\xi_{1,n,j})\frac{\firstres_{n,j}}{2}
+O_j\left(\frac{1}{(n+2)^2}\right)
\\
& =(-1)^j \left(\tan\frac{\al_j}{2}
+\tan'(\xi_{2,n,j})\frac{\firstres_{n,j}}{2}
+O_j\left(\frac{1}{(n+2)^2}\right)\right).
\end{align*}
Using the definition of $\al_j$, this can be simplified to
\[
\bigl(\tan'(\xi_{2,n,j})+(-1)^{j-1}\tanh'(\xi_{1,n,j})\bigr)\,\firstres_{n,j}
=O_j\left(\frac{1}{(n+2)^2}\right).
\]
The coefficient before $\firstres_{n,j}$ is strictly positive and bounded away from zero.
Indeed, for all $x$ from the considered domain $(j\pi/2,(j+1)\pi/2)$
we have $\tan'(x)>1$ and
\[
\tanh'(x)=\frac{1}{1+x^2}<\frac{1}{1+\frac{\pi^2}{4}}<\frac{1}{2},
\]
thus
\[
\tan'(\xi_{2,n,j})+(-1)^{j-1}\tanh'(\xi_{1,n,j})>\frac{1}{2}.
\]
Therefore $\firstres_{n,j}=O_j(1/(n+2)^2)$, which is equivalent to \eqref{eq:first_ph_asymptotics}.
The function $g$ has the following asymptotic expansion near the point $0$:
\begin{equation}\label{eq:g_near_zero}
g(x)=x^4+O(x^6).
\end{equation}
Using the formula $\la_{n,j}=g(\ph_{n,j})$
and combining \eqref{eq:first_ph_asymptotics} with \eqref{eq:g_near_zero},
we arrive at \eqref{eq:first_eigenvalues_asymptotics}.
\end{proof}

\medskip

\begin{numtest}\label{numtest:first_eigenvalue}
Denote by $\eps_{n,j}$ the absolute value of the residue in \eqref{eq:first_eigenvalues_asymptotics}:
\[
\eps_{n,j} \eqdef \left|\la_{n,j} - \left(\frac{\al_j}{n+2}\right)^4\right|.
\]
Similarly to Numerical test~\ref{numtest:inner_eigenvalues},
the exact eigenvalues $\la_{n,j}$ and the coefficients $\al_j$
are computed in high-precision arithmetic with $100$ decimal digits after the floating point.
The next table shows $\eps_{n,j}$ and $(n+2)^6 \eps_{n,j}$
corresponding to $j=1,2$ and to various values of $n$.
\[
\begin{array}{c|c|c|c|c|c}
\medstrut &
n=64 &
n=256 &
n=1024 &
n=4096 &
n=16384
\\\hline
\medstrut
\eps_{n,1} &
6.3\cdot 10^{-9} &
1.8\cdot 10^{-11} &
4.5\cdot 10^{-16} &
1.1\cdot 10^{-19} &
2.7\cdot 10^{-23}
\\\hline
\medstrut
(n+2)^6 \eps_{n,1} &
523.37 &
524.39 &
524.46 &
524.46 &
524.46
\\\hline
\medstrut
\eps_{n,2} &
1.1\cdot 10^{-7} &
3.1\cdot 10^{-11} &
7.9\cdot 10^{-15} &
2.0\cdot 10^{-18} &
4.9\cdot 10^{-22}
\\\hline
\medstrut
(n+2)^6 \eps_{n,2} &
9315.7 &
9266.9 &
9263.7 &
9263.5 &
9263.4
\end{array}
\]
Moreover, numerical experiments show that
\[
\max_{1\le n\le 100000} ((n+2)^6 \eps_{n,1}) < 524.47.
\]
\end{numtest}

\begin{rem}
Notice that formula (2.7) from \cite{BG2017pentadiagonal}
does not have the form~\eqref{eq:first_ph_asymptotics}
because the numerator $u_{1,j}$ in this formula depends on $n$ in a complicated manner.
\end{rem}

\begin{rem}
Proposition~\ref{prop:first_eigenvalues} has trivial corollaries
about the norm of the inverse matrix and the condition number:
\[
\|T_n^{-1}(g)\|_2 \sim \left(\frac{n+2}{\al_1}\right)^4,\quad
\operatorname{cond}_2(T_n(g)) \sim 16 \left(\frac{n+2}{\al_1}\right)^4,
\quad\text{as}\quad n\to\nf.
\]
\end{rem}

\begin{rem}
Proposition~\ref{prop:first_eigenvalues} is not terribly new. Parter~\cite{Parter1, Parter2} showed
that if $g_m$ is given by~\eqref{g_m}, then
\begin{equation}
\la_{n,j}=\frac{\gamma_j(m)}{(n+2)^{2m}}+o\left(\frac{1}{(n+2)^{2m}}\right)\;\;\mbox{as}\;\: n \to \infty\label{Pa}
\end{equation}
with some constant $\gamma_j(m)$ for each fixed $j$. Our proposition identifies $\gamma_1(2)$ as $\alpha_1^4$
and improves the $o(1/(n+2)^4)$ to $O(1/(n+2)^6)$.
Parter also had explicit formulas for $\gamma_j(2)$ in terms of the solutions of certain transcendental equations.
Widom~\cite{Widom1, Widom2} derived results like~\eqref{Pa}
by replacing matrices by integral operators with piecewise constant kernels and
subsequently proving the convergence of the appropriately scaled integral operators. Widom's approach
delivered the constants $\gamma_j(m)$ as the reciprocals of the eigenvalues of certain integral operators.
More about these pioneering works can be found in~\cite[pp.~256--259]{BG2005book} and in~\cite{BoWi}.
The proof of Proposition~\ref{prop:first_eigenvalues} given above is different from the ones by Parter and Widom.
\end{rem}

\begin{rem}\label{rem:n_plus_two_is_natural}
If we pass to the denominator $n+1$ in formula \eqref{eq:first_eigenvalues_asymptotics},
then it becomes more complicated:

\vspace{-5mm}
\[
\la_{n,j} = \frac{\al_j^4}{(n+1)^4}-\frac{4\al_j^4}{(n+1)^5}+O_j\left(\frac{1}{(n+2)^6}\right).
\]
This reveals that the denominator $n+2$ is objectively better when studying
the asymptotic behavior of the first eigenvalues in the example \eqref{g2}.
\end{rem}

\section{The regular four term asymptotic expansion for the example}

\begin{lem}\label{lem:false_rhs}
Let $g(x)=\left(2\sin\frac{x}{2}\right)^4$
and let $d_0,\ldots,d_4$ be the same functions as in Proposition~\ref{prop:inner_eigenvalues}.
Then, as $n \to \infty$, we have the asymptotic expansions
\begin{align}
\sum_{k=0}^3 \frac{d_k\left(\frac{j\pi}{n+2}\right)}{(n+2)^k}
=\frac{(j\pi+\eta(0))^4-\eta(0)^4}{(n+2)^4}+O\left(\frac{j^4}{(n+2)^5}\right),
\label{eq:four_term_rhs}
\\
\sum_{k=0}^4 \frac{d_k\left(\frac{j\pi}{n+2}\right)}{(n+2)^k}
=\left(\frac{j\pi+\eta(0)}{n+2}\right)^4+O\left(\frac{j^4}{(n+2)^5}\right),
\label{eq:false_rhs}
\end{align}
uniformly in $j$.
\end{lem}

\begin{proof}
By \eqref{eq:g_derivatives},
the function $g$ and its derivatives admit the following asymptotic expansions near the point $0$:
\begin{equation}\label{eq:g_and_its_derivatives_near_zero}
\begin{aligned}
&g(x)=x^4+O(x^6),\quad
g'(x)=4x^3+O(x^5),\quad
g''(x)=12x^2+O(x^4),\\
&g'''(x)=24x+O(x^3),\quad
g^{(4)}(x)=24+O(x^2)\qquad
(x\to0).
\end{aligned}
\end{equation}
Applying \eqref{eq:eta_derivatives} and taking into account that $\eta$ is smooth, we see that
\begin{equation}\label{eq:c0_c1_near_zero}
c_0(x)=x,\qquad c_1(x)=\eta(x)=\eta(0)+O(x)\qquad(x\to0)
\end{equation}
and that the functions $c_2,c_3,c_4$ are bounded.
Substituting \eqref{eq:g_and_its_derivatives_near_zero} and \eqref{eq:c0_c1_near_zero}
into the formulas \eqref{eq:d_formulas}, we get the following
expansions of $d_0(x),\ldots,d_4(x)$, as $x\to 0$:
\[
\begin{aligned}
d_0(x)&=x^4+O(x^6),\quad
d_1(x)=g'(x)c_1(x)=4x^3\eta(0)+O(x^4),\\
d_2(x)&=\frac{1}{2}g''(x)c_1^2(x)+O(x^3)=6x^2\eta^2(0)+O(x^3),\\
d_3(x)&=\frac{1}{6}g'''(x)c_1^3(x)+O(x^2)=4x\eta^3(0)+O(x^2),\\
d_4(x)&=\frac{1}{24}g^{(4)}(x)c_1^4(x)+O(x)=\eta^4(0)+O(x).
\end{aligned}
\]
Using these formulas and the binomial theorem, we arrive at \eqref{eq:false_rhs}.
Moving in \eqref{eq:false_rhs} the summand with $k=4$ to the right-hand side
we obtain \eqref{eq:four_term_rhs}.
\end{proof}

The following proposition proves Theorem~\ref{Theo 1.2}(c).

\begin{prop}\label{prop:regular_asymptotics_of_the_third_order}
Let $g(x)=\left(2\sin\frac{x}{2}\right)^4$
and $d_0,\ldots,d_3\colon[0,\pi]\to\bR$ be the functions
from the proof of Proposition~\ref{prop:inner_eigenvalues}.
Then there exists a $C >0$ such that
\begin{equation}\label{eq:regular_asymptotics_of_the_third_order}
\left|\la_{n,j}-\sum_{k=0}^3 \frac{d_k\left(\frac{j\pi}{n+2}\right)}{(n+2)^k}\right| \le \frac{C}{(n+2)^4}
\end{equation}
for all $n \in \bN$ and all $j \in \{1,\ldots,n\}$.
\end{prop}

\begin{proof}
Thanks to Proposition~\ref{prop:inner_eigenvalues}
we are left with the case $j<2\log(n+2)$.
Using \eqref{eq:main_equation_pentadiagonal},
the upper estimate \eqref{eq:rho_upper_estimate_j_small},
and the smoothness of $\eta$, we conclude that
\begin{equation}\label{eq:ph_first_rough}
\ph_{n,j}=\frac{j\pi+\eta(0)}{n+2}+O\left(\frac{j}{(n+2)^2}\right)
+O\left(\frac{e^{-2j}}{n+2}\right).
\end{equation}
From \eqref{eq:g_near_zero} we therefore obtain that
\[
\la_{n,j}=g(\ph_{n,j})
=\ph_{n,j}^4+O\left(\frac{(\log(n+2))^6}{(n+2)^6}\right)
=\ph_{n,j}^4+O\left(\frac{1}{(n+2)^4}\right).
\]
Expanding $\ph_{n,j}^4$ by the multinomial theorem and separating the main term, we get
\[
\ph_{n,j}^4
= \left(\frac{j\pi+\eta(0)}{n+2}\right)^4
+ \sum_{\substack{p,q,r\ge 0\\p+q+r=4\\p<4}}
O\left(\frac{(j\pi+\eta(0))^p j^q\,e^{-2jr}}{(n+2)^{p+2q+r}}\right).
\]
The sum over $p,q,r$ can be divided into the part with $q>0$
and the part with $q=0$ and estimated by
\[
\sum_{\substack{p,q,r\ge 0\\p+q+r=4\\q>0}}
O\left(\frac{(j\pi+\eta(0))^p j^q}{(n+2)^{4+q}}\right)
+ \sum_{\substack{p,r\ge 0\\p+r=4\\r>0}}
O\left(\frac{(j\pi+\eta(0))^p e^{-2jr}}{(n+2)^4}\right)
=O\left(\frac{1}{(n+2)^4}\right).
\]
Thus, the true asymptotic expansion of $\la_{n,j}$
under the condition $j<2\log(n+2)$ is
\begin{equation}\label{eq:small_eigenvalues_true_asymptotic_expansion_3}
\la_{n,j}= \left(\frac{j\pi+\eta(0)}{n+2}\right)^4+O\left(\frac{1}{(n+2)^4}\right).
\end{equation}
On the other hand, using~\eqref{eq:four_term_rhs}
and the fact that $j^4=O(n+2)$, we get
\begin{equation}\label{eq:rhs_3}
\sum_{k=0}^3 \frac{d_k(u_{n,j})}{(n+2)^k}
=\left(\frac{j\pi+\eta(0)}{n+2}\right)^4+O\left(\frac{1}{(n+2)^4}\right).
\end{equation}
Comparing \eqref{eq:small_eigenvalues_true_asymptotic_expansion_3}
and \eqref{eq:rhs_3}, we obtain the required result.
\end{proof}

\begin{rem}\label{rem:approximation_is_bad}
Let us again embark on the case $p=3$ and thus on Theorem~\ref{Theo 1.2}(b). 
The approximation of the first eigenvalues $\la_{n,j}$ by $\sum_{k=0}^3\frac{d_k(u_{n,j})}{(n+2)^k}$
is bad in the sense that the absolute error of this approximation
is of the same order $O_j(1/(n+2)^4)$ as the eigenvalue $\la_{n,j}$ which we want to approximate!
To state it in different terms, for each fixed $j$, the residues
\[
\omega_{n,j} := \la_{n,j}-\sum_{k=0}^3 \frac{d_k(\frac{j\pi}{n+2})}{(n+2)^k}
\]
decay at the same rate $O_j(1/(n+2)^4)$
as the eigenvalues $\la_{n,j}$ and the distances between them,
and the corresponding relative errors do not tend to zero:
\begin{equation}\label{rel:errors}
\begin{aligned}
&\frac{\omega_{n,j}}{\la_{n,j}}\to \frac{\al_j^4+\eta(0)^4-(j\pi+\eta(0))^4}{\al_j^4}\ne 0,\\
&\frac{\omega_{n,j}}{\la_{n,j+1}-\la_{n,j}}\to \frac{\al_j^4+\eta(0)^4-(j\pi+\eta(0))^4}{\al_{j+1}^4-\al_j^4}\ne 0.
\end{aligned}
\tag{REL}
\end{equation}
Compared to this, the residues of the asymptotic expansions for simple-loop symbols
(see~\cite{BBGM2015simpleloop,BGM2017relaxed}) can be bounded by
$o\left(\frac{j\,(n+1-j)}{n^2}\,\frac{1}{n^p}\right)$,
where $p$ is related with the smoothness of the symbols,
and the expression $\frac{j\,(n+1-j)}{n^2}$ is in the simple-loop case
always comparable with the distance $\la_{n,j+1}-\la_{n,j}$ between the consecutive eigenvalues,
i.e., there exist $C_1>0$ and $C_2>0$ such that
\[
C_1\,\frac{j\,(n+1-j)}{n^2} \le \la_{n,j+1} - \la_{n,j} \le C_2\,\frac{j\,(n+1-j)}{n^2}.
\]
Clearly, the quotient $\frac{|\omega_{n,j}|}{\la_{n,j+1}-\la_{n,j}}$
is a more adequate measure of the quality of the approximation
than just the absolute error $|\omega_{n,j}|$.
\end{rem}

\begin{numtest}
Denote by $\Delta_n$ the maximal error in \eqref{eq:regular_asymptotics_of_the_third_order}:
\[
\Delta_n \eqdef \max_{1\le j\le n}
\left|\la_{n,j}-\sum_{k=0}^{3} \frac{d_k(u_{n,j})}{(n+2)^k}\right|.
\]
The following table shows $\Delta_n$ and $(n+2)^4 \Delta_n$
for various values of $n$.
\[
\begin{array}{c|c|c|c|c|c}
\medstrut &
n=64 &
n=256 &
n=1024 &
n=4096 &
n=16384
\\\hline
\medstrut
\Delta_n &
7.6\cdot 10^{-6} &
3.2\cdot 10^{-8} &
1.3\cdot 10^{-10} &
5.1\cdot 10^{-13} &
2.0\cdot 10^{-15}
\\\hline
\medstrut
(n+2)^4 \Delta_n &
143.97 &
143.05 &
142.81 &
142.75 &
142.74
\end{array}
\]
According to this table, the numbers $\Delta_n$ really behave like $O(1/(n+2)^4)$.
\end{numtest}

\section{There is no regular five term asymptotic expansion for the example}
\label{sec:contradiction}

As said, Ekstr\"{o}m, Garoni, and Serra-Capizzano \cite{EGS2017areknown} conjectured
that for every infinitely smooth $2\pi$-periodic real-valued even function $g$,
strictly increasing on $[0,\pi]$,
the eigenvalues $\la_{n,j}$ of the corresponding Toeplitz matrices
admit an asymptotic expansion
of the regular form \eqref{basic} for every order $p$.

We now show that for the symbol $g(x)=\left(2\sin\frac{x}{2}\right)^4$
an asymptotic expansion
of the form \eqref{basic} cannot be true for $p=4$.
This disproves Conjecture~1 from \cite{EGS2017areknown}.

\begin{prop}\label{prop:no_regular_with_n_plus_two}
Let $g(x)=\left(2\sin\frac{x}{2}\right)^4$.
Denote by $\la_{n,1},\ldots,\la_{n,n}$
the eigenvalues of the Toeplitz matrices $T_n(g)$, written in the ascending order.
Then there do not exist continuous functions
$d_0,\ldots,d_4\colon[0,\pi]\to\bR$ and numbers $C>0$, $N\in\bN$,
such that for every $n\ge N$ and every $j \in \{1,\ldots,n\}$
\begin{equation}\label{eq:expansion_4}
\left|\la_{n,j}-\sum_{k=0}^4 \frac{d_k\left(\frac{j\pi}{n+2}\right)}{(n+2)^k}\right| \le \frac{C}{(n+2)^5}.
\end{equation}
\end{prop}

\begin{proof}
Reasoning by contradiction,
assume that there exist functions $d_0,\ldots,d_4$
and numbers $C$ and $N$ with the required properties.
Put
\[
J = \bigl\{(n,j)\in\bN^2\colon\ n\ge N,\ 2\log(n+2) \le j\le n\bigr\}.
\]
Clearly, this set $J$ asymptotically fills $[0,\pi]$ by quotients.

So, by Proposition~\ref{prop:uniqueness}, the functions $d_0,\ldots,d_4$ from \eqref{eq:expansion_4}
must be the same as the functions $d_0,\ldots,d_4$ from Proposition~\ref{prop:inner_eigenvalues}.
In other words, the asymptotic expansion \eqref{eq:la_expansion_pentadiagonal_far_from_0}
from Proposition~\ref{prop:inner_eigenvalues} holds for every pair $(n,j)$
with $n$ large enough and $j$ in $\{1,\ldots,n\}$, that is, without the restriction $j\ge 2\log(n+2)$.

Combining \eqref{eq:expansion_4} with \eqref{eq:false_rhs},
we see that for each fixed $j$ the eigenvalue $\la_{n,j}$ must have the asymptotic behavior
\begin{equation}\label{eq:conjectured_asymptotic_first_eigenvalues}
\la_{n,j}=\left(\frac{j\pi+\eta(0)}{n+2}\right)^4+O_j\left(\frac{1}{(n+2)^5}\right).
\end{equation}
Since $\eta(0)=2\arctan(1)=\frac{\pi}{2}$, we obtain for $j=1$ that
\begin{equation}\label{eq:conjectured_asymptotic_first_eigenvalue}
\la_{n,1}=\left(\frac{3\pi/2}{n+2}\right)^4+O\left(\frac{1}{(n+2)^5}\right),
\end{equation}
which contradicts Proposition~\ref{prop:first_eigenvalues} because $3\pi/2 \ne \al_1$.
\end{proof}

\begin{rem} \label{rem:another_end_of_the_proof}
Here is an alternative way to finish the proof of Proposition~\ref{prop:no_regular_with_n_plus_two}.
After having formula \eqref{eq:conjectured_asymptotic_first_eigenvalues},
we obtain the following hypothetical asymptotic relation between two first eigenvalues:
\[
\lim_{n\to\nf}\left((n+2) \left(\la_{n,2}^{1/4}-\la_{n,1}^{1/4}\right)\right)
=(2\pi+\eta(0))-(\pi+\eta(0))=\pi.
\]
But this contradicts Proposition~\ref{prop:first_eigenvalues}, according to which
\[
\lim_{n\to\nf}\left((n+2) \left(\la_{n,2}^{1/4}-\la_{n,1}^{1/4}\right)\right)=\al_2-\al_1<\pi.
\]
In this reasoning we do not use the value $\eta(0)$.
\end{rem}

\noindent
{\bf Proof of Theorem 1.2.} The existence of the asymptotic expansions~\eqref{12aaa} follows from
Proposition~\ref{prop:inner_eigenvalues}, its uniqueness is a consequence of Proposition~\ref{prop:uniqueness},
formula~\eqref{12ccc} was established in Proposition~\ref{prop:regular_asymptotics_of_the_third_order},
and the impossibility of~\eqref{12bbb} is just
Proposition~\ref{prop:no_regular_with_n_plus_two}. \hfill $\square$

\begin{rem}
Note that Proposition~\ref{prop:no_regular_with_n_plus_two} is actually stronger than the second part of
Theorem ~1.2. Namely, Theorem~1.2 states that~\eqref{12bbb} cannot hold with the functions $d_1, \ldots, d_4$
appearing in~\eqref{12aaa}. Proposition~\ref{prop:no_regular_with_n_plus_two} tells us that~\eqref{12bbb}
is also impossible for any other choice of continuous functions $d_1, \ldots, d_4$. The reason is of course
Proposition~\ref{prop:uniqueness}.
\end{rem}

\noindent
{\bf Proof of Theorem 1.1.}
The functions $d_0,d_1,\ldots$ from Proposition~\ref{prop:inner_eigenvalues}
are infinitely smooth on $[0,\pi]$,
and thus, by Remark~\ref{rem:change_denominator},
the expansion \eqref{eq:la_expansion_pentadiagonal_far_from_0} with $p=4$
can be rewritten in the form \eqref{eq:expansion_4_n_plus_one}
with some infinitely smooth functions $f_0,\ldots,f_4$.
So, \eqref{eq:expansion_4_n_plus_one}
is true for all $(n,j)$ satisfying that $2\log(n+2)\le j\le n$.

Contrary to what we want, assume that there are $f_0,\ldots,f_4$, $C$, and $N$
as in the statement of Theorem~\ref{thm:no_regular_with_n_plus_one}.
Then, by Proposition~\ref{prop:uniqueness}, the functions $f_0,\ldots,f_4$
are the same as those in the previous paragraph.
In particular, $f_0,\ldots,f_4$ must be infinitely smooth.
In this case, the asymptotic expansion \eqref{eq:expansion_4_n_plus_one}
can be rewritten in powers of $1/(n+2)$
and is true for all $n$ and $j$ with $n\ge N$ and $1\le j\le n$.
This contradicts Proposition~\ref{prop:no_regular_with_n_plus_two}. \hfill $\square$

\medskip
We conclude with a conjecture about
the eigenvalues of Toeplitz matrices generated by \eqref{g_m}.

\begin{conj}
Let $g_m(x)=\left(2\sin\frac{x}{2}\right)^{2m}$ with an integer $m \ge 3$. If $p \le 2m-1$, there are
$N_p\in\bN$ and $D_p>0$ such that
\begin{equation}\label{last}
\left|\la_{n,j}-\sum_{k=0}^p \frac{d_k\left(\frac{j\pi}{n+2}\right)}{(n+2)^k}\right| \le \frac{D_p}{(n+2)^{p+1}}
\end{equation}
for all $n\ge N_p$ and all $j$ in $\{1,\ldots,n\}$. For $p=2m$, inequality~\eqref{last} does not hold
for all sufficiently large $n$ and all $1 \le j \le n$, but it holds for
for all sufficiently large $n$ and all $j$ not too close to $1$, say, for $(\log(n+2))^2 \le j \le n$.
\end{conj}

\bigskip\noindent
Mauricio~Barrera\\
CINVESTAV\\
Departamento de Matem\'aticas\\
Apartado Postal 07360\\
Ciudad de M\'exico\\
Mexico\\
\texttt{mabarrera@math.cinvestav.mx}

\bigskip\noindent
{Albrecht~B\"{o}ttcher\\
Technische Universit\"{a}t Chemnitz\\
Fakult\"{a}t f\"{u}r Mathematik\\
09107 Chemnitz\\
Germany\\
\texttt{aboettch@mathematik.tu-chemnitz.de}

\bigskip\noindent
Sergei~M.~Grudsky\\
CINVESTAV\\
Departamento de Matem\'aticas\\
Apartado Postal 07360\\
Ciudad de M\'exico\\
Mexico\\
\texttt{grudsky@math.cinvestav.mx}

\bigskip\noindent
Egor~A.~Maximenko\\
Instituto Polit\'ecnico Nacional\\
Escuela Superior de F\'isica y Matem\'aticas\\
Apartado Postal 07730\\
Ciudad de M\'exico\\
Mexico\\
\texttt{emaximenko@ipn.mx}


\begin{thebibliography}{10}

\bibitem{BG2017pentadiagonal}
M. Barrera and S.M. Grudsky:
{\em Asymptotics of eigenvalues for pentadiagonal symmetric Toeplitz matrices.}
Operator Theory: Adv. and Appl. 259, 179--212 (2017).
\doi{10.1007/978-3-319-49182-0\_7}

\bibitem{BBGM2015simpleloop}
J.M. Bogoya, A. B\"{o}ttcher, S.M. Grudsky, and E.A. Maximenko:
{\em Eigenvalues of Hermitian Toeplitz matrices with smooth simple-loop symbols.}
J. Math. Analysis Appl. 422, 1308--1334 (2015).
\doi{10.1016/j.jmaa.2014.09.057}

\bibitem{BBGM2015maximum}
J.M. Bogoya, A. B\"{o}ttcher, S.M. Grudsky, and E.A. Maximenko:
{\em Maximum norm versions of the Szeg\H{o} and Avram--Parter theorems for Toeplitz matrices.}
J. Approx. Theory 196, 79--100 (2015).
\doi{10.1016/j.jat.2015.03.003}

\bibitem{BBM2016convergence}
J.M. Bogoya, A. B\"{o}ttcher, E.A. Maximenko:
{\em From convergence in distribution to uniform convergence.}
Bolet\'{i}n de la Sociedad Matem\'{a}tica Mexicana 22:2, 695--710 (2016).
\doi{10.1007/s40590-016-0105-y}

\bibitem{BGM2017relaxed}
J.M. Bogoya, S.M. Grudsky, and E.A. Maximenko:
{\em Eigenvalues of Hermitian Toeplitz matrices generated by simple-loop symbols with relaxed smoothness.}
Operator Theory: Adv. and Appl. 259, 179--212 (2017).
\doi{10.1007/978-3-319-49182-0\_11}

\bibitem{BGM2010inside}
A. B\"{o}ttcher, S.M. Grudsky, and E.A. Maksimenko:
{\em Inside the eigenvalues of certain Hermitian Toeplitz band matrices.}
J. Comput. Appl. Math. 233, 2245--2264 (2010).
\doi{10.1016/j.cam.2009.10.010}

\bibitem{BG2005book}
A. B\"{o}ttcher and S.M. Grudsky: {\em Spectral Properties of Banded Toeplitz Matrices.}
SIAM, Philadelphia (2005).
\doi{10.1137/1.9780898717853}

\bibitem{BoWi}
A. B\"ottcher and H. Widom:
{\em From Toeplitz eigenvalues through Green's  kernels to higher-order Wirtinger-Sobolev inequalities.}
Operator Theory:  Adv. and Appl. 171, 73--87 (2006).
\doi{10.1007/978-3-7643-7980-3\_4}

\bibitem{DIK2012}
P. Deift, A. Its, and I. Krasovsky:
{\em Eigenvalues of Toeplitz matrices in the bulk of the spectrum.}
Bull. Inst. Math. Acad. Sin. (N.S.) 7, 437--461 (2012).
\myurl{http://web.math.sinica.edu.tw/bulletin\_ns/20124/2012401.pdf}

\bibitem{EGS2017areknown}
S.-E. Ekstr\"{o}m, C. Garoni, and S. Serra-Capizzano:
{\em Are the eigenvalues of banded symmetric Toeplitz matrices known in almost closed form?}
Experimental Mathematics, 1--10 (2017).
\doi{10.1080/10586458.2017.1320241}

\bibitem{Elouafi2014}
M. Elouafi:
{\em On a relationship between Chebyshev polynomials and Toeplitz determinants.}
Appl. Math. Comput. 229, 27--33 (2014).
\doi{10.1016/j.amc.2013.12.029}

\bibitem{Parter1}
S.V. Parter: 
{\em Extreme eigenvalues of Toeplitz forms and applications to elliptic difference equations.}
Trans. Amer. Math. Soc. 99, 153--192 (1961).
\doi{10.2307/1993449}

\bibitem{Parter2}
S.V. Parter: 
{\em On the extreme eigenvalues of truncated Toeplitz matrices.}
Bull. Amer. Math. Soc. 67, 191--196 (1961).
\doi{10.1090/S0002-9904-1961-10563-6}

\bibitem{Trench1993interlacement}
W.F. Trench: 
{\em Interlacement of the even and odd spectra of real symmetric Toeplitz matrices.}
Linear Alg. Appl. 195, 59--68 (1993).
\doi{10.1016/0024-3795(93)90256-N}

\bibitem{Widom1}
H. Widom: 
{\em Extreme eigenvalues of translation kernels.}
Trans. Amer. Math. Soc. 88, 491--522 (1958).
\doi{10.1090/S0002-9947-1961-0138980-4}

\bibitem{Widom2}
H. Widom: 
{\em Extreme eigenvalues of $N$-dimensional convolution operators.}
Trans. Amer. Math. Soc. 106, 391--414 (1963).
\doi{10.2307/1993750}

\end{thebibliography}
\end{document}